\documentclass[12pt,reqno,a4paper]{amsart}
\usepackage{amssymb}

\usepackage{amscd}

\usepackage{csquotes}

\newcommand{\RNum}[1]{\uppercase\expandafter{\romannumeral #1\relax}}

\usepackage{graphicx}

\usepackage[T2A]{fontenc}
\usepackage[utf8]{inputenc}
\usepackage[russian,english]{babel}
\usepackage{yfonts}

\input{ell-sympl.def}

\usepackage{mdwlist}

\usepackage{hyperref}
\usepackage{blindtext}
\usepackage{scrextend}

\usepackage{tikz}
\usetikzlibrary{matrix,arrows,calc}
\usepackage{tikz-cd}
\usepackage{graphicx}
\usepackage{caption}
\usepackage{subcaption}
\usepackage{wrapfig}

\usepackage{enumitem}

\usepackage{amsmath}

\usepackage{xcolor}

\numberwithin{equation}{section}

\begin{document}

\myitemmargin 
\baselineskip =14.9pt plus 1.8pt

\def\conf{{\sf{Conf}}}
\def\o{{\omega}}

\title[Elliptic diffeomorphisms of symplectic $4$-manifolds]%
{Elliptic diffeomorphisms\\[3pt] of symplectic $4$-manifolds}


\author{Vsevolod Shevchishin}
\address{Faculty of Mathematics and Computer Science\\
University of Warmia and Mazury\\
ul.~Słoneczna 54, 10-710 Olsztyn, Poland
}
\email{shevchishin@gmail.com}
\thanks{}

\author{Gleb Smirnov}
\address{ETH Z\"{u}rich \\ R\"{a}mistrasse 101 \\ 8092 Z\"{u}rich, Switzerland}
\curraddr{}
\email{gleb.smirnov@math.ethz.ch}
\thanks{}

\subjclass[2010]{Primary: }

\date{}

\dedicatory{}

\begin{abstract} 
We show that symplectically embedded $(-1)$-tori give rise to certain elements in the symplectic mapping class group of $4$-manifolds. An example is given where such elements are proved to be of infinite order.
\end{abstract}

\maketitle

\tableofcontents

\setcounter{section}{-1}
\section{Introduction}\label{intro}

Let $(X,\omega)$ be a closed symplectic 4-manifold. Denote by $\pi_0 \symp(X,\omega)$ the group of symplectic diffeomorphisms of $X$ modulo symplectic isotopy. Let us consider the forgetful homomorphism
\[
\pi_0 \symp(X,\omega) \to \pi_0 \diff(X).
\]

Here $\pi_0 \diff(X)$ denotes the smooth mapping class group for $X$. It is known this homomorphism is not necessary injective. If $\Sigma$ is a smooth Lagrangian sphere in $X$, then there exists a symplectomorphism $T_{\Sigma} : X \to X$, called \slsf{symplectic Dehn twist along} $\Sigma$, such that $T_\Sigma^2$ is smoothly isotopic to the identity. In his thesis \cite{SeiTh}, Seidel proved that in many cases $T_\Sigma^2$ is not symplectically isotopic to the identity.
He than proved that for certain $K3$ surfaces containing two Lagrangian spheres $\Sigma_1$ and $\Sigma_2$, the element $T_{\Sigma_1}^2$ has infinite order, and hence the forgetful homomorphism has infinite kernel. The reader is invited to look at \cite{Sei2} for a detailed
description of symplectic Dehn twists.

Somewhat later, Biran and Giroux introduced different symplectomorphisms, namely the fibered Dehn twists, among which one can find smoothly yet not symplectically trivial maps. In fact, Seidel's Dehn twist is a particular case of a fibered Dehn twist. Suppose that that $X$ admits a separating contact type hypersurface $P$ carrying a free $S^1$-action in $P \times [0,1]$ that preserves the contact form on $P$. Then one can define the \slsf{fibered Dehn twist} as
\[
T_{P} \colon P \times [0,1] \to P \times [0,1], \quad (x,t) \to (x \cdot [f(t)\,\text{\normalfont{mod}}\,2 \pi], t),
\]
where a function $f \colon [0,1] \to \rr$ equals $2\pi$ near $t = 0$ and $0$ near $t = 1$. As $T_P$ is a symplectomorphism of $P \times [0,1]$ that is the identity near the boundary of $P \times [0,1]$, it can be extended to be a symplectomorphism of the whole $X$. We refer the reader to \cite{R-D-O,U} for an extensive study of fibered Dehn twists.

Given that it is easy to find a separating contact hypersurface, fibered Dehn twists make an effective tool to construct symplectomorphisms of a given $4$-manifold (and of a higher-dimensional manifold, for that matter.) But even though a plethora of results has been obtained in symplectic mapping class groups (see e.g. \cite{Ab-McD, Bu, Anj, Anj-Gr,Anj-Lec, Ev, La-Pin, LiJ-LiT-Wu,Ton, Sei1, Sei3, Wen}), it remains hard to detect nontriviality of symplectomorphisms.
\medskip%

In this paper we introduce and study a new type of symplectomorphisms for
4-manifolds. In short, our construction is as follows. Let  $(X,\omega_0)$  be a symplectic $4$-manifold which contains a symplectically embedded torus  $C\subset X$ of self-intersection $(-1)$. In particular, $\mu := \int_C \omega_0>0$. 
We construct a family of symplectic forms $\omega_t$ on $X$ in the cohomology class $[\omega_t]$ such that $\int_{C} \omega_t = \mu - t$. We show that such a family exists for $t$ large enough for $C$ to have a negative symplectic area.

For each $t$ we construct an $\omega_t$-symplectomorphism $T_{C} : X \to X$, called the \slsf{elliptic twist along} $C$. 
As smooth maps, those symplectomorphisms $T_{C}$ for different $t$ are isotopic,
so we can think of $T_C$ as a single diffeomorphism defined up to isotopy.

We then study whether or not these elliptic twist are symplectically isotopic to the identity. It appears that it is so in the case when $\int_C \omega_t > 0$. In particular, $T_C$ is always smoothly isotopic to the identity. As we shall see below, it is not so in the case $\int_C \omega_t \le 0$, and $T_C$ could be non-trivial.

Let $Y$ be a tubular neighbourhood of $C$ in $X$.
Then $\del Y$ is a separating contact hypersurface in $X$, which 
carries a free $S^1$-action. One can pick a symplectic form $\wt{\o}_0$ on $X$ such that $\omega_0|_{X-Y} = \wt{\omega}_0|_{X-Y}$ and $\int_{C} \wt{\o}_0 \leq 0$. We conjecture that, for $(X,\wt{\omega}_0)$, the fibered Dehn twist associated to $\del Y$ is symplectically isotopic to $T_C$.
\smallskip %

Our first result is an example of a $4$-manifold $X$ and a $(-1)$-torus $C$ in it, where the elliptic twist $T_C$ turns out to be always symplectically trivial.

\begin{thm}\label{thm-vanish}
Let $(X,\omega)$ be a symplectic ruled 4-manifold diffeomorphic to the total space of the non-trivial $S^{2}$-bundle over $T^2$; we denote it by $S^2 \tilde{\times} T^2$ for short. Then

\sli there is a symplectic form $\omega_0$ on $X$ which admits an $\omega_0$-symplectic $(-1)$-torus $C \subset X$, and the elliptic twist $T_{C}$ is well-defined.

\slii the forgetful homomorphism
$\pi_0 \symp(X,\omega) \to \pi_0 \diff(X)$ is injective
for any symplectic form $\omega$.
In particular, the elliptic twist $T_{C}$ 
is always symplectically isotopic to the identity.
\end{thm}
The injectivity property claimed in part \slii was proved previously by McDuff for $S^2 \times T^2$, see \cite{McD-B}.
We thus cover the remaining non-spin case and, therefore, prove the so-called symplectic isotopy conjecture for elliptic ruled surfaces, see \slsf{Problem 14} in \cite{McD-Sa-1}.

The main result of this note shows that it is possible for an elliptic twist to contribute nontrivially to a symplectic mapping class group.

\begin{thm}\label{thm-notvanish} 
Let $Z$ be $S^2 \tilde{\times} T^2 \, \# \, \overline{\mathbb{CP}^2}$.
There exist a symplectic form $\omega$ on $Z$ and three $\o$-symplectic $(-1)$-tori $C_1, C_2$, and $C_3$ in $Z$ such that the elliptic twists $T_{C_{i}}$ are well-defined and none of them is symplectically isotopic to the identity; each $T_{C_{i}}$ has infinite order in the symplectic mapping class group.
\end{thm}

Our proof follows closely to the ideas introduced by Abreu-McDuff in \cite{Ab-McD} and McDuff in \cite{McD-B}.

The main technique we use in the proof is Gromov's theory of pseudoholomorphic curves. 
This theory involves various Banach manifolds and constructions with them. Dealing with them we often pretend to be in the finite-dimensional case. We refer the reader to the book \cite{Iv-Sh-1}  and articles \cite{Iv-Sh-2,Iv-Sh-3} for a comprehensive analytic setup to Gromov's theory of pseudoholomorphic curves. Of course,  reader is free to address to any of numerous alternative sources and expositions of the theory such as \cite{McD-Sa-3} or the seminal paper \cite{Gro}. 

\state Acknowledgements.  
We are deeply indebted to Boris Dubrovin, Yakov Eliashberg, and Viatcheslav Kharlamov for a number of useful suggestions which were crucial for the present exposition of this paper.
Part of this note was significantly improved during our stay at the University of Pisa and the Humboldt University of Berlin, and we are very grateful to Paolo Lisca and Klaus Mohnke for numerous discussions and for the wonderful research environment they provided. We also would like to thank Rafael Torres for reading the manuscript and pointing out certain inconsistencies.
Special thanks to Dasha Alexeeva for sending us a preprint of her thesis.
Finally, we are grateful to the referee for a positive comment on our paper, for her/his constructive and thorough criticism, and for the tremendous amount of work he/she did reviewing the manuscript. The second author was supported by an ETH Fellowship.

\section{Construction of the elliptic twist.}

\subsection{Elliptic twist}
\label{Ell-tw-1}
Let $(X,\omega)$ be a symplectic $4$-manifold, and let $C$ 
be an embedded symplectic $(-1)$-torus in $X$. 
We let $\Omega(X,\o)$ to denote the space of symplectic forms on $X$ that are isotopic to $\o$, 
and let $\calj(X,\Omega)$ to denote the space of almost-complex structures for which there 
exists a taming form in $\Omega(X,\o)$.

Pick an almost-complex structure $J_0 \in \calj(X,\Omega)$ for which $C$ 
is pseudoholomorphic. One thinks of $J_0$ as a point of the subspace $\cald_{[C]} \subset \calj(X,\Omega)$ of those almost-complex structures which admit a smooth pseudoholomorphic curve in the class $[C]$. 
In what follows, we refer to $\cald_{[C]}$ as the \slsf{elliptic divisorial locus} for the class $[C]$.
The term \emph{divisorial locus} is taken from the fact that in some neighbourhood of $J_0$ the subspace $\cald_{[C]}$ locally behaves as a submanifold of $\calj(X,\Omega)$ of real codimension $2$, see e.g. \cite{Iv-Sh-1}.

Let $\Delta \subset \calj(X,\Omega)$ be a small disc transverally intersecting $\cald_{[C]}$ precisely at $J_0$, and let $J \colon [0,1] \to \Delta$ be the boundary of $\Delta$.
We will make the following assumption:
\smallskip%

$\sf{(A)}$ There exists a class $\xi \in \sfh_2(X;\rr)$, $\xi \cdot [C] \leq 0$ such that 
every $J(t) \in \del \Delta$ is tamed by some symplectic form $\theta_t$, $[\theta_t] = \xi$. 
\smallskip%

One can arrange $\theta_t$ so that they depend smoothly on $t$. 
Moser isotopy then gives us a path of diffeomorphisms $f_t \colon X \to X$, $f_t^{*} \theta_t = \theta_0$.
Now $f_1$ is a symplectomorphism of $(X,\theta_0)$. 
We call $f_1$ the \slsf{elliptic twist along $C$} and use the notation $T_C$ for it.

As we will explain below (see \refsubsection{sec}), for $T_C$ to give a non-trivial element in 
$\pi_0\symp(X, \theta_0)$, it is necessary that $[J(t)] \in \pi_1(\calj(X,\Theta))$ is non-trivial; 
here $\Theta$ stands for the space of symplectic forms on $X$ that are isotopic to $\theta_0$. 
We emphasize that $J(t)$ is contractible in $\calj(X,\Omega)$; 
thus, $T_C$ is trivial for $(X,\omega)$.

Assumption $\sf{(A)}$ always holds, though we do not prove it in the full generality.
But we shall consider a series of $4$-manifolds for which the assumption is 
easy to verify. Let $(X,\o)$ be a symplectic $4$-manifold, and let $C$ be a symplectic 
torus of self-intersection number $\sf{0}$. 
Take an $\o$-tamed almost-complex structure on $X$ for which $C$ becomes pseudoholomorphic, 
and then perturb this structure slightly to make it integrable in some tubular 
neighbourhood of $C$. More precisely, we want a sufficiently small neighbourhood of $C$ to admit 
an elliptic fibration with $C$ being a multiple fiber of multiplity $ m > 1$.

Let $T^2$ be an elliptic curve $\cc/\zz \tau_1 \oplus \zz \tau2$, where $(\tau_1,\tau_2)$ 
form a basis for $\cc(u)$ as a real vector space, and let $\Delta$ be a complex disc with a local parameter $z$.
The neighbourhood of $C$ in $X$ is biholomorphic to the quotient 
$\Delta \times T^2/\sim$, where $(z,u) \sim (z e^{2 \pi i/m }, u + \tau_1/m)$. 
Here $C$ is given by the equation $\left\{ z = 0 \right\}$.

Blowing-up $X$ at a point $(0,u_0)$, we get a manifold $Z$ which 
contains a smooth elliptic curve in the class $[C] - \bfe$ 
(the strict transform of $C$.)
Here $\bfe$ stands for the homology class of the exceptional line.
Unless $z_0 = 0$, the blow-up of $X$ at $(z_0,u_0)$ does not contain such a curve,
since it contains one in the class $m [C] - \bfe$ (the strict transform of 
$\left\{ z = z_0 \right\}$, which we denote by $C_m$.)

Pick a taming symplectic form $\o_0$ on $Z$. 
Clearly, the form satisfies 
\[
\int_{[C]-\bfe} \o > 0.
\]

Let $Z(t)$ be the blow-up of $X$ at $(R\, e^{i t},u_0)$.
Observe that the complex structures on $Z(t)$ are $\o$-tamed for $R$ sufficiently small.
Using the deflation (see \refsubsection{sec-econom}) along $C_m$, 
we deform $\o$ on $Z(t)$ into a symplectic form $\theta_t$ for which
\[
\int_{[C_m]} \theta_t = \varepsilon
\]
for $\varepsilon$ positive arbitrary small. 
Implementing deflation does not violate the taming condition for $Z_t$.
Being performed in a small neighbourhood of $C_m$, the deflation does not affect 
the symplectic area of $C$, see \cite{Bu}. Since
\[
\int_{[C]-\bfe} \theta_t = \varepsilon - (m-1) \int_{[C]} \omega,
\]
one may take $m$ sufficiently large to make the area of $[C] - \bfe$ as negative as desired.
We have now verified $\sf{(A)}$ for the family $Z(t)$.

\subsection{The Abreu-McDuff framework.}\label{sec} Let $\diff_{0}(X)$ be the identity component of the diffeomorphism group of $X$ and 
$\Omega_{}(X,\omega)$ be the space of all symplectic forms on $X$ that are isotopic to $\omega$. We have a natural transitive action of $\diff_0(X)$ on $\Omega(X,\omega)$. So we get a principle fiber bundle
\begin{equation}\label{FundDiag}
\symp(X,\omega) \cap \diff_{0}(X) \to \diff_{0}(X) \to \Omega(X,\omega),
\end{equation}
where the last arrow stands for the map
\[
\varphi: \diff_{0}(X) \to \Omega(X,\omega) 
\qquad\text{with}\qquad 
\varphi: f \mapsto f_{*} \omega.
\]
To shorten notation, 
we put:
\[
\symp^*(X,\omega) := \symp(X,\omega) \cap \diff_{0}(X)
\]
Following \cite{Kh}, we consider an exact sequence of homotopy groups 
\[
\ldots \to \pi_1(\diff_0(X)) \xrightarrow{\;\varphi_{*}\;} \pi_1(\Omega(X,\omega)) \xrightarrow{\;\del\;} \pi_0(\symp^*(X,\omega)) \xrar{\;\;} 1= \pi_0(\diff_0(X)).
\]

Let $\calj(X,\Omega)$ be the space of those almost-complex structures $J$ on $X$ 
for which there exists a taming symplectic form $\omega_{J} \in \Omega(X,\omega)$. 
It is easy to see that $\calj(X,\Omega)$ is connected. Let $J_0$ be some $\omega$-tamed almost-complex structure.
It was shown by McDuff, see \slsf{Lemma 2.1} in \cite{McD-B}, that there exists a homotopy equivalence 
$\psi \colon \Omega(X,\omega) \to \calj(X,\Omega)$ 
for which the diagram
\begin{equation}\label{mcduff-map}
\begin{tikzcd}
\diff_{0}(X) \arrow{r}{\varphi} \arrow{dr}{\nu} & \Omega(X,\omega) \arrow{d}{\psi} \\
{} & \calj(X,\Omega)
\end{tikzcd}
\end{equation}
commutes. Here $\nu \colon \diff_{0}(X) \to \calj(X,\Omega)$ is given by $\nu: f \mapsto f_{*} J_0$.

Following the fundamental idea of Gromov's theory \cite{Gro} we study the space $\calj(X,\Omega)$ rather than $\Omega(X,\omega)$. We see from the following diagram
\begin{equation}\label{d-kron}
\begin{CD}
\ldots @>>> \pi_1(\diff_0(X)) @>{\varphi_{*}}>>  \pi_1(\Omega(X,\omega)) @>{\del}>> \pi_0(\symp^*(X,\omega)) @>>> 0 \\
&& @V{\text{\slsf{id}}}VV @V{\psi_{*}}VV \\
\ldots @>>> \pi_1(\diff_0(X)) @>{\nu_{*}}>> \pi_1(\calj(X,\Omega)),
\end{CD}
\end{equation}
that each loop in $\calj(X,\Omega)$ contributes non-trivially to the symplectic mapping class group of $X$, provided this loop does not come from $\diff_0(X)$.
We will use diagram \eqref{d-kron} to prove both \refthm{thm-vanish} and \refthm{thm-notvanish}. The reader is referred to \cite{McD-B} for more extensive discussion of the topic.

In what follows we work with a slightly bigger space $\calj^{k}(X,\omega)$ of $C^{k}$-smooth almost-complex structures.
Here and below \enquote{$C^{k}$-smoothness} means some $C^{k,\alpha}$-smoothness with $0 < \alpha < 1$ and $k$ natural sufficiently large.
The reason to do this is that the space $\calj^{k}(X,\omega)$ is a Banach manifold, while the space of $C^{\infty}$-smooth structures $\calj(X,\Omega)$ is merely Fréchet.
What we prove for $\pi_{i}(\calj^{k}(X,\omega))$ works perfectly for $\pi_{i}(\calj(X,\Omega))$ because the inclusion
$\calj(X,\Omega) \hookrightarrow \calj^{k}(X,\omega)$ induces the weak homotopy equivalence 
$\pi_i( \calj(X,\Omega) ) \to \pi_i( \calj^{k}(X,\omega) )$.

\subsection{Symplectic economics.}\label{sec-econom}

Here we give a brief description of the inflation technique developed by Lalonde-McDuff \cite{La-McD, McD-B}, and a generalization of this procedure given by Bu\c se, see \cite{Bu}.
\begin{thm}[Inflation]\label{thm-Infla}
Let $J$ be an $\omega_0$-tamed almost
complex structure on a symplectic $4$-manifold $(X, \omega_0)$ that
admits an embedded $J$-holomorphic curve $C$ with $[C]\cdot [C] \ge 0$. 
Then there is a family $\omega_s, s \ge 0,$ of symplectic forms  that all tame $J$ and have cohomology class 
\[
[\omega_s] = [\omega_0] + s\, \pd([C]),
\]
where $\pd([C])$ is Poincar\'e dual to $[C]$.
\end{thm}
For negative curves a somewhat reverse procedure exists, called negative inflation or deflation.

\begin{thm}[Deflation]\label{thm-Defla}
Let $J$ be an $\omega_0$-tamed almost
complex structure on a symplectic $4$-manifold $(X, \omega_0)$ that
admits an embedded $J$-holomorphic curve $C$ with $[C]\cdot [C] = -m$. Then there is a
family $\omega_s$ of symplectic forms  that all tame $J$ and have
cohomology class 
\[
[\omega_s] = [\omega_0] + s\, \pd([C])
\]
for all $0 \leq s < \dfrac{\omega_0([C])}{m}$.
\end{thm}

\section{Elliptic geometrically ruled surfaces}

\subsection{General remarks.} A complex surface $X$ is called ruled if there exists a holomorphic map $\pi:X\to Y$ to a Riemann surface $Y$ such that each fiber $\pi^{-1}(y)$ is a rational curve; if, in addition, each fiber is irreducible, then $X$ is called geometrically ruled.
A ruled surface is obtained by blowing up a geometrically ruled surface. Note however that a geometrically ruled surface need not be minimal (the blow up of $\mathbb{CP}^2$, denoted by $\mathbb{CP}^2\, \# \, \overline{\mathbb{CP}^2}$, 
is the unique example of a geometrically ruled surface that is not a minimal one). Unless otherwise noted, all ruled surfaces are assumed to be geometrically ruled. 
One can speak of the genus of the ruled surface $X$, meaning thereby the genus of $Y$. We thus have rational ruled surfaces, elliptic ruled surfaces and so on.

Up to diffeomorphism, there are two total spaces of orientable $S^2$-bundles over a Riemann surface: the product $S^2 \times Y$ and the non-trivial bundle $S^2 \tilde{\times} Y$. The product bundle admits sections $Y_{2\,k}$ of even self-intersection number $[Y_{2\,k}]^2 = 2\,k$, and the non-trivial bundle admits sections $Y_{2\,k+1}$ of odd self-intersection number $[Y_{2\,k+1}]^2 = 2\,k+1$. We will choose the basis $\bfy = [Y_0], \bfs = [\text{pt} \times S^2]$ for $\sfh_2(S^2 \times Y; \zz)$, and use the basis $\bfy_{-} = [Y_{-1}], \bfy_{+} = [Y_{1}]$ for $\sfh_2(S^2 \tilde{\times} Y; \zz)$. To simplify notations, we denote both the classes $\bfs$ and $\bfy_{+} - \bfy_{-}$, which are the fiber classes of the ruling, by $\bff$. Further, the class $\bfy_{+} + \bfy_{-}$, which is a class for a bisection of $X$, will be of particular interest for us, and will be widely used in forthcoming computations; we denote this class by $\bfb$.
Throughout this paper we will freely identify homology and cohomology by Poincaré duality.

Clearly, we have $[Y_{2\,k}] = \bfy + k\,\bff$ and $[Y_{2\,k + 1}] = \bfy_{+} + k\,(\bfy_{+} + \bfy_{-})$. This can be seen by evaluating the intersection forms for these 4-manifolds on the given basis:
$$
\displaystyle{
\mathcal{Q}_{S^2 \times Y} = 
\begin{pmatrix}
0 & 1\\ 
1 & 0
\end{pmatrix},\quad 
\mathcal{Q}_{S^2 \tilde{\times} Y} = 
\begin{pmatrix}
1 & 0\\ 
0 & -1
\end{pmatrix}.
}
$$
Observe that these forms are non-isomorphic. That is why the manifolds $S^2 \times Y$ and $S^2 \tilde{\times} Y$ are non-diffeomorphic. One more way to express the difference between them is to note that the product $S^2 \times Y$ is a spin 4-manifold, but $S^2 \tilde{\times} Y$ is not spin. 
Note that after blowing up one point, they become diffeomorphic:
$S^2 \times Y\,  \# \, \overline{\mathbb{CP}^2} \simeq S^2 \tilde{\times} Y\, \# \, \overline{\mathbb{CP}^2}$.

This section is mainly about the non-spin elliptic ruled surface $S^2 \tilde{\times} T^2$. When studying this manifold we sometimes use the notations $\bft_{+}$ and $\bft_{-}$ instead of $\bfy_{+}$ and $\bfy_{-}$ for the standard homology basis in $\sfh^2(S^2 \tilde{\times} T^2;\zz)$.

From the viewpoint of complex geometry every such $X$ is a holomorphic $\cp^1$-bundle over a Riemann surface $Y$ whose structure group is $\pgl(2,\cc)$. Biholomorphic classification of ruled surfaces is well understood, at least for low values of the genus. Below we
recall a part of the classification of elliptic ruled surfaces given by Atiyah in \cite{At-2}; this being the first step towards understanding the almost-complex geometry of these surfaces. We also provide a short summary of Suwa's results: \sli an explicit construction of a complex analytic family of ruled surfaces, where one can see the jump phenomenon of complex structures,
see \refsubsection{sec-family}, \slii an examination of those complex surfaces which are both ruled and admit an elliptic pencil, see \slsf{Theorem \ref{thm-fibering}}.

In what follows we will use a formula for the first Chern class of 
a geometrically ruled surface. In terms of $\bfy, \bfs, \bfy_{\pm}$, it becomes
\begin{equation}\label{eq-chern}
c_1(S^2 \times Y) = 2\,\bfy + \chi(Y)\,\bfs,\quad 
c_1(S^2 \tilde{\times} Y) = (1+\chi(Y))\,\bfy_{+} + (1-\chi(Y))\,\bfy_{-}.
\end{equation}
The symplectic geometry of ruled surfaces has been extensively studied by many authors \cite{Li-Li, Li-Liu-1, Li-Liu-2, AGK, Sh-4, H-Iv}. Ruled surfaces are of great interest from the symplectic point of view mainly because of the following significant result due to Lalonde-McDuff, see \cite{La-McD, McD-6}.
\begin{thm}[The classification of ruled 4-manifolds]\label{thm-mcduff}
Let $X$ be oriented diffeomorphic to a minimal rational or ruled surface, and let $\xi \in \sfh^2(X)$. Then there is a symplectic form (even a Kähler one) on $X$ in the class $\xi$ iff $\xi^2 > 0$. Moreover, any two symplectic forms in the class $\xi$ are diffeomorphic.
\end{thm}
Thus all symplectic properties of ruled surfaces depend only on the cohomology class of a symplectic form. 

Our main interest is to study symplectic $(-1)$-tori in $X$ and the corresponding elliptic twists. It is easy to prove that, except for $S^2 \tilde{\times} T^2$, there are no symplectic $(-1)$-tori in ruled surfaces. 
For a suitable symplectic form
the homology class $\bft_{-} \in \sfh_2(S^2 \tilde{\times} T^2; \zz)$ can be represented by a symplectic $(-1)$-torus, but none of the other classes of $\sfh_2(S^2 \tilde{\times} T^2; \zz)$ can.  

Let $(X, \omega_{\mu})$ be a symplectic ruled 4-manifold $(S^2 \tilde{\times} T^2, \omega_{\mu})$, where $\omega_{\mu}$ is a symplectic structure of the cohomology class
$
[\omega_{\mu}] = \bft_{+} - \mu \bft_{-}
$, $\mu \in (-1,1)$.
By \slsf{Theorem \ref{thm-mcduff}} $(X,\omega_{\mu})$ is well-defined up to symplectomorphism. As promised in the introduction, we will prove that $\pi_0 \symp^*(X,\omega_{\mu})$ is trivial. 
Here and in \refsubsection{sec-vanish} we abbreviate $\Omega(X,\o_{\mu})$ to $\Omega_{\mu}$.

Given $\mu > 0$, the elliptic divisorial locus is contained in $\calj(X,\Omega_{\mu})$. 
Thus, each loop linked to the locus is contractible in $\calj(X,\Omega_{\mu})$. 
As such, we do not expect any non-trivial elliptic twists in this case.
Following McDuff \cite{McD-B}, we will show that the group $\symp^*(X,\omega_{\mu})$ coincides with a group of certain diffeomorphisms, see \slsf{Lemma \ref{lem-epiclemma}}; the latter group can be proved to be connected by standard topological techniques, see \slsf{Proposition \ref{prop-nonspin}}.

When $\mu \leq 0$, the elliptic divisorial locus $\cald_{\bft_{-}}$ is no longer included in 
$\calj(X,\Omega_{\mu})$. The geometry of this divisorial locus is studied below in
\refsubsection{sec-vanish}, and particularly it is proved that: 
\sli Assumption $\sf{(A)}$ is satisfied for each loop linked to $\cald_{\bft_{-}}$; 
hence, $(X,\omega_{\mu})$ admits certain elliptic twists, see \lemma{lem-suka}. 
\slii The symplectic mapping class group $\symp^*(X,\omega_{\mu})$ is generated by elliptic twists
coming from $\cald_{\bft_{-}}$.
\sliii Each of them is symplectically isotopic to the identity, see \lemma{lem-unexpectedlemma}.

\subsection{Classification of complex surfaces ruled over elliptic curves.}\label{sec-complex}
Here we very briefly describe possible complex structures on elliptic ruled surfaces and study some of their properties.

Let $X$ be diffeomorphic to either $S^2 \times Y^2$ or $S^2 \tilde{\times} Y^2$. The Enriques-Kodaira classification of complex surfaces (see e.g.\cite{BHPV}) 
ensures the following:
\begin{enumerate}
\item\label{enum-alg} Every complex surface $X$ of this diffeomorphism type is algebraic and hence K\"ahler.
\item\label{enum-rul} Every such complex surface $X$ is ruled, i.e. there exists a holomorphic map $\pi : X\to Y$ such that $Y$ is a complex curve, and each fiber $\pi\inv(y)$ is an irreducible rational curve. Note that, with the single exception of $\mathbb{CP}^1 \times \mathbb{CP}^1$, a ruled surface admits at most one ruling.
\end{enumerate}

It was shown by Atiyah \cite{At-2} that every holomorphic $\cp^1$-bundle over a curve $Y$ with structure group the projective group $\pgl(2,\cc)$ admits a holomorphic section, and hence the structure group of such bundle can be reduced to the affine group $\Aff(1,\cc) \subset \pgl(2,\cc)$.

All of what was said so far applied for any ruled surface, irrespective of genus. Keep in mind, however, that everything below is for genus one surfaces. It was Atiyah who gave a classification of ruled surfaces with base an elliptic curve. The description presented here is taken from \cite{Sw}.
\begin{thm}[Atiyah]
Every holomorphic $\cp^1$-bundle with structure group $\pgl(2,\cc)$ over an elliptic curve is isomorphic to preciesly one of the following:
\begin{enumerate}
\item[\sli] a bundle associated to a principal $\cc^{*}$-bundle of nonpositive degree,
\item[\slii] a bundle $\mathcal{A}$, defined below, having structure group $\Aff(1,\cc)$, and
\item[\sliii] a bundle $\mathcal{A}^{Spin}$, having structure group $\Aff(1,\cc)$.
\end{enumerate}
\end{thm}
We shall proceed with a little discussion of these bundles:

\medskip%
\sli We first describe those $\pgl(2,\cc)$-bundles whose structure group reduces to $\cc^{*}$.
Let $y \in Y$ be a point on the curve $Y$, and let $\{V_0, V_1\}$ be an open cover of $Y$ such that $V_0 = Y \setminus \{y\}$ and $V_1$ is a small neighbourhood of $y$, so the domain $V_0 \cap V_1 =: \hat{V}$ is a punctured disc. We choose a multivalued coordinate $u$ on $Y$ centered at $y$.

A surface $X_{k}$ associated to the line bundle $\scro(k\,y)$ (or if desired, a $\cc^{*}$-bundle) can be described as follows:
\[
X_k := \left( V_0 \times \cp^1 \right) \cup \left( V_1 \times \cp^1 \right)/\sim,
\]
where $(u,z_0) \in V_0 \times \cp^1$ and $(u,z_1) \in V_1 \times \cp^1$ are identified iff $u \in \hat{V},\ z_1 = z_0 u^k$. Here $z_0, z_1$ are inhomogeneous coordinates on the two copies of $\cp^1$.
Clearly, the biholomorphism $
(u,z_0) \to (u,z_0^{-1}),\,(u,z_1) \to (u,z_{1}^{-1})
$
maps $X_k$ to $X_{-k}$. Thus it is sufficient to consider only values of $k$ that are nonpositive.

There is a natural $\cc^{*}$-action on $X_k$ via $g \cdot (z_0,u) := (gz_0,u),\, g \cdot (z_1,u) := (gz_1,u)$
for each $g \in \cc^{*}$. The fixed point set of this action consists of two mutually disjoint sections $Y_{k}$ and $Y_{-k}$ defined respectively by $z_0 = z_1 = 0$ and $z_0 = z_1 = \infty$. We have $[Y_{k}]^2 = k$ and $[Y_{-k}]^2 = -k$.

It is very well known that any line bundle $L$ of degree $\deg(L) = k \neq 0$ is isomorphic to $\scro(k\,y)$ for some $y \in Y$. Thus all the ruled surfaces associated with line bundles of non-zero degree $k$ are biholomorphic to one and the same surface $X_k$.
On the other hand, the parity of the degree of the underlying line bundle is a topological invariant of a ruled surface. More precisely, a ruled surface $X$ associated with a line bundle $L$ is diffeomorphic to $Y \times S^2$ for $\deg(L)$ even, and to $Y \tilde{\times} S^2$ for $\deg(L)$ odd.

\medskip%
\slii Again, we start with an explicit description of the ruled surface $X_{\mathcal{A}}$ associated with the affine bundle $\mathcal{A}$. Let $\{V_0, V_1, \hat{V}\}$ be the open cover of $Y$ as before, $u$ be a coordinate on $Y$ centered at $y$, and $z_0,z_1$ be fiber coordinates. Define
\[
X_{\mathcal{A}} := \left( V_0 \times \cp^1 \right) \cup \left( V_1 \times \cp^1 \right)/\sim,
\]
where $(z_0,u) \sim (z_0,u)$ for $u \in \hat{V}$ and $z_0 = z_1 u + u^{-1}$.

There is an obvious section $Y_1$ defined by the equation $z_0 = z_1 = \infty$, but in contrast to $\cc^{*}$-bundles, the surface $X_{\mathcal{A}}$ contains no section disjoint from that one.
This can be shown by means of direct computation in local coordinates, but one easily deduce this from \slsf{Theorem} \ref{thm-fibering} below.

We will make repeated use of the following geometric characterization of $X_{\mathcal{A}}$, whose proof is given in \cite{Sw}, see \slsf{Theorem 5}.
\begin{thm}\label{thm-fibering}
The surface $X_{\mathcal{A}}$ associated with the affine bundle $\mathcal{A}$
admits an elliptic fibration over $\cp^1$; the general fiber is a smooth elliptic curve in the class $\displaystyle{ 2\,\bft_{+} + 2\,\bft_{-} }$ and there are three multiple fibers each having the property that the underlying reduced subvariety is a smooth elliptic curve in the class $\displaystyle{ \bft_{+} + \bft_{-} }$. There are no other singular fibers.
\end{thm}
The following corollary will be used later. The reader is invited to look at \cite{McD-D} for the definition of the Gromov invariants and some examples of their computation.
\begin{corol}\label{cor-gw}
$\gr(\bfy_{+} + \bfy_{-}) = 3$.
\end{corol}
\begin{proof}
There are no smooth curves in $X_{\cala}$, other than the multiple fibers, that are in the class $\bft_{+} + \bft_{-}$. Each multiple fiber contributes $\pm1$ to $\gr(\bft_{+} + \bft_{-})$, for their normal bundles are holomorphically non-trivial, see \slsf{§ 1.7} in \cite{McD-D}. If the complex structure is integrable, then each {\itshape non-multiple-covered} torus should appear with sign $(+1)$, see \cite{Tb}. \qed  
\end{proof}

\medskip%
Based on this theorem, Suwa then gives another construction of $X_{\mathcal{A}}$. We mention this construction here because it appears to have interest for the sequel.

Let $Y \cong \cc/\zz \tau_1 \oplus \zz \tau_2$ be an elliptic curve, 
and $u$ be a multivalued coordinate on $Y$.
Define $X_{\cala}$ to be a quotient space of $\cp^1 \times Y/\calg$, where $\mathcal{G} \cong \zz_2 \oplus \zz_2$ is generated by the following involutions
\[
\displaystyle{
(z, u) \to \left(-z, u + \frac{\tau_1}{2} \right),\quad (z, u) \to \left( \frac{1}{z}, u + \frac{\tau_2}{2} \right).
}
\]
The surface obtained is elliptic ruled and is non-spin; see \cite{Sw}, where the latter is 
proved by constructing a section for $X_{\mathcal{A}}$ of odd self-intersection number, 
see also \slsf{Exercises 6.13 and 6.14} in \cite{McD-Sa-1}.

The elliptic fibration of $X_{\cala}$ mentioned in \refthm{thm-fibering} comes from the 
$\calg$-invariant function
\[
f(z,u)=\frac{1}{2}\left( z^2 + \frac{1}{z^2} \right),
\]
whose values are regular for all but three points of $\cp^1$. For a regular point, when $z \neq \left\{ -1,1,\infty \right\}$, the fiber $f^{-1}(z)$ is an elliptic curve in the class $2 (\bft_{+} + \bft_{-})$, whereas each of the three singular fibers is a curve in the class $\bft_{+} + \bft_{-}$.

There is an obvious action of the complex torus $\mathcal{T} \cong Y$ on $\cp^1 \times Y$ by translations. This action commutes with that of $\calg$. Hence, $\mathcal{T}$ acts also on $X_{\cala}$. As the function $f$ is $\mathcal{T}$-invariant, so are the fibers $f^{-1}(z)$, $z \in \cp^1$ of our elliptic fibration; they are, in fact, simply the orbits of the action. 
Although $\mathcal{T}$ acts effectively on $X_{\cala}$, it does not act freely; the isotropy groups of the singular fibers correspond to the three pairwise different order two subgroups of $\mathcal{T}$. For instance, for $(z,u) \in f^{-1}(\infty)$, the stabilizer is $z \to -z$. 

As each fiber $f^{-1}(z)$ is the torus, it gives a homomorphism $\sfh_1(f^{-1}(z);\zz) \to \sfh_1(X_{\cala};\zz)$ between the two copies of $\zz^2$. 
To see what this homomorphism is for the multiple fibers, we regard $X_{\cala}$ as a ruled surface 
over $Y^{\prime} \cong \cc/\zz (\tau_{1}/2) \oplus \zz (\tau_2/2)$. Then the multiple fibers 
appear as bisections, double covering of $Y^{\prime}$. Note that there is a one-to-one correspondence 
between the double covering of $Y^{\prime}$ and the index $2$ subgroups of $\sfh_1(Y^{\prime};\zz)$.
This implies that, for the singular fibers $f^{-1}(z), z = \left\{ -1,1,\infty \right\}$, the images of 
$\sfh_1(f^{-1}(z);\zz) \to \sfh_1(X_{\cala};\zz)$ correspond to three pairwise different index $2$ 
subgroups of $\sfh_1(X_{\cala};\zz) \cong \sfh_1(Y^{\prime};\zz)$.
\medskip%

\sliii The ruled surface associated to $\mathcal{A}^{Spin}$ is diffeomorphic to $S^2 \times T^2$, thus it will not be discussed in this note, but see \cite{Sw}.

\medskip%
Summarizing our above observations, we see that $X \cong S^2 \tilde{\times} T^2$ admits countably many diffeomorphism classes of complex structures. These structures are as follows:
\begin{itemize}
\item The structures $J \in \scrj_{1-2\,k},\,k > 0$, such that the ruled surface $(X,J)$ contains a section of self-intersection number $1-2\,k$; these are all biholomorphic to $X_{1-2\,k}$.

\item The type $\cala$ structures $J \in \scrj_{\mathcal{A}}$ such that 
the ruled surface $(X,J)$ contains no sections of negative self-intersection number but does contain a triple of smooth bisections; these are all biholomorphic to $X_{\mathcal{A}}$.
\end{itemize}

\subsection{One family of ruled surfaces over elliptic base.}\label{sec-family}
Here is a construction of a one-parametric complex-analytic family $p : \scrx \to \cc$ of non-spin elliptic ruled surfaces, such that the surfaces $p^{-1}(t),\,t \neq 0$, are biholomorphic to $X_{\mathcal{A}}$ and $p^{-1}(0) \cong X_{-1}$.

As before, we take a point $y$ on $Y$, let $u$ be a coordinate of the center $y$, and put $\{V_0, V_1, \hat{V}\}$ to be an open cover for $Y$ such that $V_0 := Y \setminus \{y\}$, $V_1$ is a small neighbourhood of $y$, and $\hat{V} := V_0 \cap V_1$. Further, let $\Delta$ be a complex plane, and let $t$ be a coordinate on it.

We construct the complex 3-manifold $\scrx$ by patching $\Delta \times V_0 \times \cp^1$ and $\Delta \times V_1 \times \cp^1$ in such a way that $(t,z_0,u) \sim (t,z_1,u)$ for $u \in \hat{V}$ and $z_0 = z_1u + tu^{-1}$.

The preimage of $0$ and $1$ under the natural projection $p : \scrx \to \Delta$ are biholomorphic respectively to $X_{-1}$ and $X_{\mathcal{A}}$. In fact, it is not hard to see that for each $t \neq 0$, the surface $p^{-1}(t)$ is biholomorphic to $X_{\mathcal{A}}$ as well.
One way to prove this is to use the $\cc^{*}$-action on $\scrx$
\[
g \cdot (t, z_0, u) := (tg, gz_0,u),\quad g \cdot (t,uz_1,u) := (tg, gz_1, u)\quad \text{\normalfont{for each $g \in \cc$}}.
\]
This proves even more than we desired, namely, that there exists a $\cc^{*}$-action on $\scrx$ such that for each $g \in \cc^{*}$ we get a commutative diagram
$$
\displaystyle{
\begin{CD}
\scrx @>{\cdot g}>> \scrx \\
@V{p}VV @VV{p}V \\
\cc @>>{\cdot g}> \cc,
\end{CD}
}
$$
where $\scrx \xrightarrow {\cdot g} \scrx$ denotes the biholomorphism induced by $g \in \cc^{*}$. 

The construction of the complex-analytic family $\scrx$ is due to Suwa, see \cite{Sw}, though the existence of the $\cc^{*}$-action was not mentioned in Suwa's paper. Let us summarize his result in a theorem.

\subsection{Embedded curves and almost-complex structures.}\label{sec-emb}

In \refsection{sec-complex} the classification for non-spin elliptic ruled surfaces was given. It turns out that this classification can be extended to the almost-complex geometry of $S^2 \tilde{\times} T^2$.

Let $X$ be diffeomorphic to $S^2 \tilde{\times} T^2$, and let $\calj(X)$ be the space of almost-complex structures on $X$ that are tamed by some symplectic form; the symplectic forms need not be the same.
Here we use the short notation $\calj$ for $\calj(X)$.

Given $k > 0$, let $\calj_{1-2k}(X)$ (we will abbreviate it to $\calj_{1-2k}$) be the subset of $J \in \calj$ consisting of elements that admit a smooth irreducible $J$-holomorphic elliptic curve in the class $\bft_{+}-k\,\bff$. It is well known that $\calj_{1-2k}$ forms a subvariety of $\calj$ of real codimension $2{\cdot}(2k-1)$, see e.g. \slsf{Corollary 8.2.3} in \cite{Iv-Sh-1}. 

Further, define $\calj_{\cala}(X)$ (or $\calj_{\cala}$, for short) be the subset $J \in \calj$ of those element for which there exists a smooth irreducible $J$-holomorphic elliptic curve in the class $\bfb$. 

By straightforward computations one can show that the sets $\calj_{1-2k}$ are mutually disjoint, and each $\calj_{1-2k}$ is disjoint from $\calj_{\cala}$. 
Further, it is not hard to see that $\calj_{-1} \subset \overline{\calj}_{\cala}$ and $\calj_{1-2(k+1)} \subset \overline{\calj}_{1-2k}$, where $\overline{\calj}_{1-2k}$ is for the closure of $\calj_{1-2k}$. A less trivial fact
is that
\begin{equation}\label{eq-decom}
\displaystyle{
\calj = \calj_{\cala} \bigsqcup \calj_{-1} \bigsqcup \calj_{-3} \bigsqcup \calj_{-5} \ldots,
}
\end{equation}
which can be also stated as follows.
\begin{prop}[cf. \slsf{Lemma 4.2} in \cite{McD-B}]\label{prop-smoothcurves}
Let $(X,\omega)$ be a symplectic ruled 4-manifold diffeomorphic to $S^2 \tilde{\times} T^2$.
Then every $\omega$-tamed almost-complex structure $J$ admits a smooth irreducible $J$-holomorphic representative in either $\bfb$ or $\bft_{+}-k\,\bff$ for some $k > 0$.
\end{prop}
\begin{proof}
The proof is analogous to that of \slsf{Lemma 4.2} in \cite{McD-B}. Observe that the expected codimension for the class $\bfb$ is zero. By \slsf{Lemma \ref{cor-gw}} we have $\gr(\bft_{+} + \bft_{-}) > 0$. Hence, $\calj_{\cala}$ is an open dense subset of $\calj$, and, thanks to the Gromov compactness theorem, for each $J \in \calj$ the class $\bfb$ has at least one  $J$-holomorphic representative, possibly singular, reducible or having multiple components.

By virtue of \slsf{Theorem \ref{thm-ruled}}, no matter what $J$ was chosen, our manifold $X$ admits the smooth $J$-holomorphic ruling $\pi$ by rational curves in the class $\bff$. 

Since $\bfb \cdot \bff > 0$, it follows from positivity of intersections that any $J$-holomorphic representative $B$ of the class $\bfb$ must either intersect a $J$-holomorphic fiber of $\pi$ or must contain this fiber completely.

\medskip%
\sla First assume that $B$ is irreducible. Then it is of genus not greater than $1$ because of the adjunction formula. This curve is of genus $1$ because every spherical homology class of $X$ is proportional to $\bff$.  We now can apply the adjunction formula one more time to conclude that $B$ is smooth, i.e. $J \in \calj_{\cala}$.

\medskip%
\slb The curve $B$ is reducible but contains no irreducible components which are the fibers of $\pi$. Then it contains precisely two components $B_1$ and $B_2$, since $\bfb \cdot \bff = 2$. Both the curves $B_1$ and $B_2$ are smooth sections of $\pi$, and hence $[B_i] = \bft_{+} + k_i \bff$, $i=1,2$. Since $[B_1] + [B_2] = \bfb$, it follows that $k_1 + k_2 = -1$, and hence either $k_1$ or $k_2$ is negative. Thus we have that either $B_1$ or $B_2$ is a smooth $J$-holomorphic section of negative self-intersection index.

\medskip%
\slc If some of the irreducible components of $B$ are in the fibers class $\bff$, then one can apply arguments similar to that used in \sla and \slb to prove that the part $B^{\prime}$ of $B$ which contains no fiber components has a section of negative self-intersection index as a component. \qed
\end{proof}

\subsection{Rulings and almost-complex structures.}

Let $X$ be a ruled surface equipped with a ruling $\pi : X \to Y$, and let $J$ be an almost-complex structure on $X$. We shall say that $J$ \slsf{is compatible with the ruling} $\pi : X \to Y$ if each fiber $\pi^{-1}(y)$ is $J$-holomorphic. 

We wish to express our thanks to D.\,Alexeeva \cite{Al} for sharing her proof of the following statement.
\begin{prop}\label{prop-dasha}
Let $\mathcal{J}(X,\pi)$ be the space of almost-complex structures on $X$ compatible with $\pi$.

\sli $\mathcal{J}(X,\pi)$ is contractible.

\slii Any structure $J \in \mathcal{J}(X,\pi)$, as well as any compact family $J_t \in \mathcal{J}(X,\pi)$, is tamed by some symplectic form.
\end{prop}
\begin{proof}
\sli Let be $\jj(\rr^4,\rr^2)$ be the space of linear maps $J : \rr^4 \to \rr^2$ such that $J^2 = -\text{\slsf{id}}$ and $J(\rr^2) = \rr^2$, i.e. it is the space of linear complex structures preserving $\rr^2$.
In addition, we assume $\rr^4$ and $\rr^2$ are both oriented and each $J \in \jj(\rr^4,\rr^2)$ induces the given orientations for both $\rr^4$ and $\rr^2$.
We now prove the space $\jj(\rr^4,\rr^2)$ is contractible.

Indeed, let us take $J \in \jj(\rr^4,\rr^2)$. Fix two vectors $e_1 \in \rr^2$ and $e_2 \in \rr^4 \setminus \rr^2$. The vectors $e_1$ and $Je_1$ form a positively oriented basis for $\rr^2$. Therefore $Je_1$ is in the upper half-plane for $e_1$. Further, the vectors $e_1,Je_1,e_2,Je_2$ form a positively oriented basis for $\rr^4$. Therefore $Je_2$ is in the upper half-space for the hyperplane spanned on $e_1,Je_1,e_2$.

We see that the space $\jj(\rr^4,\rr^2)$ is homeomorphic to the direct product of two half-spaces, and hence it is for sure contractible.

To finish the proof of \sli we consider the subbundle $V_{x} := \ker \text{\normalfont{d}}\,\pi(x) \subset T_{x}X,\,x\in X$, of the tangent bundle $TX$ of $X$. Every $J \in \calj(X,\pi)$ is a section of the bundle $\jj(TX,V) \to X$ whose fiber over $x \in X$ is the space $\jj(T_{x}X,V_x)$. Since the fibers of $\jj(TX,V)$ are contractible; it follows that the space of section for $\jj(TX,V)$ is contractible as well.

\medskip%
\slii Again, we start with some linear algebra. Let $V$ be a 2-subspace of $W \cong \rr^4$, and let $J \in \jj(W,V)$. Choose a 2-form $\tau \in \Lambda^2(W)$ such that the restriction $\tau|_{V} \in \Lambda^2(V)$ of $\tau$ to $V$ is positive with respect to the $J$-orientation of $V$, i.e. $\tau(\xi,J\xi) > 0$. Clearly, the subspace $H := \ker \tau \subset W$ is a complement to $V$. Further, let $\sigma \in \Lambda^2(V)$ be any 2-form such that $\sigma|_{V}$ vanishes, but $\sigma|_{H}$ does not. If $H$ is given the orientation induced by $\sigma$, then the $J$-orientation of $W$ agrees with that defined by the direct sum decomposition $W \cong V \oplus H$. We now prove that $J$ is tamed by $\tau + K\, \sigma$ for $K > 0$ sufficiently large.

It is easy to show that there exists a basis $e_1,e_2 \in V,\,e_3,e_4 \in H$ for $W$ such that $J$ takes the form
\[
J = \begin{pmatrix}
0 & -1 & 1 & 0\\
1 &  0 & 0 &-1\\
0 &  0 & 0 &-1\\
0 &  0 & 1 & 0
\end{pmatrix}.
\]
The matrix $\Omega$ of $\tau + K\, \sigma$ with respect to this basis is block-diagonal, say
\[
\Omega = \begin{pmatrix}
0 & 1 & 0 & 0\\
-1& 0 & 0 & 0\\
0 & 0 & 0 & K\,\sigma + \ldots\\
0 & 0 & -K\,\sigma + \ldots & 0
\end{pmatrix}\,\, \text{\normalfont{for} $\sigma > 0$.} 
\]
It remains to check that the matrix $\Omega J$ is positive definite, i.e. $(\xi, \Omega J \xi) > 0$. A matrix is positive definite iff its symmetrization is positive definite.
It is straightforward to check that $\Omega J + (\Omega J)^{t}$ is of that kind for $K$ large enough.

Let us go back to the ruled surface $X$. The theorem of Thurston \cite{Th} (see also \slsf{Theorem 6.3} in \cite{McD-Sa-1}) ensures the existence of a closed 2-form $\tau$ on $X$ such that the restrictions of $\tau$ to each fiber $\pi^{-1}(y)$ is non-degenerate. Choose an area form $\sigma$ on $Y$.
By the same reasoning as before, any $J \in \mathcal{J}(X,\pi)$ is tamed by $\tau + K\,\pi^{*}\sigma$ for $K$ large enough. \qed
\end{proof}

\medskip%
The following theorem by McDuff motivates the study of compatible almost-complex structures, see \slsf{Lemma 4.1} in \cite{McD-B}.
\begin{thm}\label{thm-ruled}
Let $X$ be an irrational ruled surface, and let $J \in \calj(X)$. Then there exists a unique ruling $\pi : X \to Y$ such that $J \in \calj(X,\pi)$.
\end{thm}

\subsection{Diffeomorphisms.}
Let $X$ be diffeomorphic to either $T^2 \times S^2$ or $T^2 \tilde{\times} S^2$, and let $\pi : X \to Y$ be a smooth ruling. Further, let $\fol(X)$ be the space of all smooth foliations of $X$ by spheres in the fiber class $\bff$ (the class $\bff$ generates $\pi_2(X)$ and, therefore, it is 
the only class that can be the fiber class of an $S^2$-fibration.)

The group $\diff(X)$ acts transitively on $\fol(X)$ as well as the group $\diff_0(X)$ acts transitively on a connected component $\fol_0(X)$ of $\fol(X)$. This gives rise to a fibration sequence
$$
\displaystyle{
\mathcal{D} \cap \diff_0(X) \to \diff_0(X) \to \fol_0(X),
}
$$
where $\mathcal{{D}}$ is the group of fiberwise diffeomorphisms of $X$. By the definition of $\mathcal{D}$ there exists a projection homomorphism $\tau : \mathcal{D} \to \diff(T^2)$ such that for every $F \in \mathcal{D}$ we have a commutative diagram
$$
\displaystyle{
\begin{CD}
X @>F>> X \\
@V{\pi}VV @VV{\pi}V \\
T^2 @>>{\tau(F)}> T^2,
\end{CD}
}
$$
which induces the corresponding commutative diagram for homology
$$
\displaystyle{
\begin{CD}
\sfh_1(X;\zz) @>F_{*}>> \sfh_1(X;\zz) \\
@V{\pi_{*}}VV @VV{\pi_{*}}V \\
\sfh_1(T^2;\zz) @>>{\tau(F)_{*}}> \sfh_1(T^2;\zz).
\end{CD}
}
$$
Notice that $\tau(F)$ is isotopic to the identity if only if $\tau(F)_{*} = \text{\slsf{id}}$. Since $\pi_*$ is an isomorphism, it follows that the subgroup $\mathcal{D} \cap \diff_0(X)$ of $\mathcal{D}$ is mapped by $\tau$ to $\diff_0(T^2)$, so we end up with the restricted projection homomorphism
$$
\displaystyle{
\tau : \mathcal{D} \cap \diff_0(X) \to \diff_0(T^2).
}
$$
Since we shall exclusively be considering this restricted homomorphism,
we use the same notation $\tau$ for this. 

Given an isotopy $f_t \in \diff_0(T^2),\, f_0 = \text{\slsf{id}}$, one can lift it to an isotopy $F_t \in \mathcal{D} \cap \diff_0(X),\, F_0 = \text{\slsf{id}}$ such that $\tau(F_t) = f_t$. This immediately implies that the inclusion $\ker \tau \subset \mathcal{D} \cap \diff_0(X)$
induces an epimorphism 
\begin{equation}\label{eq-kertau}
\pi_{0}(\ker \tau) \to \pi_0(\mathcal{D} \cap \diff_0(X)).
\end{equation}
Because of this property we would like to look at the group $\ker \tau$ in more detail, but first introduce some useful notion.

Let $X$ be a smooth manifold, and let $f$ be a self-diffeomorphism $X$. Define the \slsf{mapping torus} $T(X,f)$ as the quotient of $X \times [0,1]$ by the identification $(x,1)\sim(f(x),0)$.
For the diffeomorphism $f$ to be isotopic to identity it is necessary to have the mapping torus diffeomorphic to $T(X,\text{\slsf{id}}) \cong X \times S^1$.

Let us go back to the group $\ker \tau$ that consists of bundle automorphisms of $\pi : X \to T^2$. Let $F \in \ker \tau$ be a bundle automorphism of $\pi$, and let $\gamma$ be a simple closed curve on $T^2$. By $F_\gamma$ denote the restriction of $F$ to $\pi^{-1}(\gamma) \cong S^1 \times S^2$. The mapping torus $T(\pi^{-1}(\gamma), F_\gamma)$ is either diffeomorphic to $S^2 \times T$ or $S^2 \tilde{\times} T$. In the later case we shall say that the automorphism $F$ is \slsf{twisted} along $\gamma$.
\begin{lem}\label{lem-obstr}
Let $X$ be diffeomorphic to either $T^2 \times S^2$ or $T^2 \tilde{\times} S^2$, and let $F \in \ker \tau$. Then $F$ is isotopic to the identity through $\ker \tau$ iff $T^2$
contains no curve for $F$ to be twisted along.
\end{lem}
\begin{proof}
The $2$-torus $T^2$ has a cell structure with one cell, $2$ 1-cells, and one 2-cell. Clearly, $F$ can be isotopically deformed to $\text{\slsf{id}}$ over the 0-skeleton of $T^2$.
The obstruction for extending this isotopy to the 1-skeleton of $T^2$ is a well-defined cohomology class $c(F) \in \sfh^1(T^2;\zz_2)$; the obstruction cochain $c(F)$ is the cochain whose value on a 1-cell $e$ equals 1 if $F$ is twisted along $e$ and 0 otherwise. It is evident that $c(F)$ is a cocycle.

By assumption $c(F) = 0$. Consequently there is an extension of our isotopy to an isotopy over a neighbourhood of the 1-skeleton of $T^2$, but such an isotopy always can be extended to the rest of $T^2$. \qed
\end{proof}
A short way of represent the issue algebraically is by means of the \slsf{obstruction homomorphism}
\[
c : \ker \tau \to \sfh^1(X;\zz_2)
\]
defined in the lemma; any two elements $F,G \in \ker \tau$ are isotopic to each other through $\ker \tau$ iff $c(F) = c(G)$.
\begin{lem}\label{lem-twist}
Let X be diffeomorphic to $S^2 \times T^2$, and let $F \in \ker \tau$, then $T^2$ contains no curve for $F$ to be twisted along. This means that the obstruction homomorphism vanishes.
\end{lem}
\begin{proof}
The converse would imply that the mapping torus $T(X,F)$ is not spin, but $T(X,\text{\slsf{id}}) \cong S^2 \times T^2 \times S^1$ is a spin 5-manifold. \qed 
\end{proof}

\smallskip%
The following result is due to McDuff \cite{McD-B}, but the proof follows by combining \slsf{Lemma \ref{lem-twist}} with \slsf{Lemma \ref{lem-obstr}}.
\begin{prop}\label{prop-spin}
Let X be diffeomorphic to $S^2 \times T^2$, then the group $\mathcal{D} \cap \diff_0(X)$ is connected.
\end{prop}

\smallskip%
In what follows we need a non-spin analogue of this Proposition for the case of elliptic ruled surfaces.
\begin{prop}\label{prop-nonspin}
Let X be diffeomorphic to $S^2 \tilde{\times} T^2$, then the group $\mathcal{D} \cap \diff_0(X)$ is connected.
\end{prop}
\begin{proof}
Fix any cocycle $c \in \sfh^1(X;\zz_2) \cong \zz_2 \oplus \zz_2$, then we claim there exists $F \in \ker \tau$ such that $c(F) = c$ and, moreover, $F$ is isotopic to $\text{\slsf{id}}$ through diffeomorphisms $\mathcal{D} \cap \diff_0(X)$. It follows from Suwa's model, see \refsubsection{sec-complex}, that the automorphism group for the complex ruled surface $X_\mathcal{A}$ contains the complex torus $\mathcal{T}$ as a subgroup. By construction, it is clear that $\mathcal{T}$ is a subgroup of $\mathcal{D} \cap \diff_0(X)$. Besides that, the 2-torsion subgroup $\mathcal{T}_2 \cong \zz_2 \oplus \zz_2$ of $\mathcal{T}$ is a subgroup of $\ker \tau$. We trust the reader to check $\mathcal{T}_2$ is mapped isomorphically by the obstruction homomorphism to $\sfh^1(T^2;\zz_2)$.

The algebra behind this argument is expressed by a commutative diagram
\[
\displaystyle{
\begin{CD}
\mathcal{T}_2 @>i>> \ker \tau @>j>> \mathcal{D} \cap \diff_0(X) \\
@VVV @VVV @VVV \\
\pi_0(\mathcal{T}_2) @>i_{*}>> \pi_0(\ker \tau) @>j_{*}>> \pi_0(\mathcal{D} \cap \diff_0(X)),
\end{CD}
}
\]
where $i_{*}$ is an isomorphism, $j_{*} \circ i_{*}$ is zero, and therefore $j_{*}$ is zero as well. But we already know that $j_{*}$ is an isomorphism, and hence $\pi_0(\mathcal{D} \cap \diff_0(X))$ is trivial. \qed
\end{proof}

\subsection{Vanishing of elliptic twists.}\label{sec-vanish}

Here is the part where a proof of \slsf{Theorem \ref{thm-vanish}} comes. We split it into a few pieces.
Let $X$ be the symplectic ruled 4-manifold $(S^2 \tilde{\times} T^2, \omega_{\mu})$, $
[\omega_{\mu}] = \bft_{+} - \mu\, \bft_{-}$, 
and let $\Omega_{\mu}$ be the space of symplectic forms on $X$ that are isotopic to $\o_{\mu}$.
Here we work with the connected component of $\calj(X)$ that contains $\calj(X,\Omega_{\mu})$; the same applies to $\calj_{\cala}$ and $\calj_{1-2k}$.

\begin{lem}\label{lem-wc}
$\calj_{\cala} \subset \calj(X,\Omega_{\mu})$ for every $\mu \in (-1,1)$.
\end{lem}
\begin{proof}
For every $J \in \calj_{\cala}$ we take any symplectic form $\omega$ such that $J$ is $\omega$-tamed. Then inflate $\omega$ along the classes $\bfy_{+} - \bfy_{-}$ and $\bfy_{+} + \bfy_{-}$, and then rescale it. \qed
\end{proof}
\begin{lem}\label{lem-suka}
$\calj_{\cala} = \calj(X,\Omega_{\mu})$ for every $\mu \in (-1,0]$.
\end{lem}
\begin{proof}
It is clear that $\calj(X,\Omega_{\mu})$ does not contain the structures $\calj_{1-2k}$ for $\mu \in (-1,0]$, and hence by \eqref{eq-decom} and \slsf{Lemma \ref{lem-wc}} the proof follows. \qed
\end{proof}

\smallskip%
This means that there is no topology change for the space $\calj(X,\Omega_{\mu})$ when $\mu$ is being varied in $(-1,0]$. In particular,
\[
\pi_1(\calj(X,\Omega_{\mu})) = \pi_1(\calj_{\cala}(X))\quad \text{\normalfont{for}}\ \mu \in (-1,0].
\]
\begin{lem}\label{lem-after0}
$\calj_{1-2\,k} \subset \calj(X,\Omega_{\mu})$ iff $\mu \in \left(1-\dfrac{1}{k},1\right)$.
\end{lem}
\begin{proof}
The ``only if'' part is obvious, while the ``if'' can be proved by deflating along $\bfy_{+}-k\,(\bfy_{+}-\bfy_{-})$ and inflating along $\bfy_{+}-\bfy_{-}$. \qed 
\end{proof}

\smallskip%
Combining \slsf{Lemma \ref{lem-wc}} with \slsf{Lemma \ref{lem-after0}}, as well as the fact that the higher codimension submanifolds $\calj_{1-2\,k}$, $k \geq 2$ do not affect the fundamental group of $\calj(X,\Omega_{\mu})$, we see that there is no topology change in $\pi_1(\calj(X,\Omega_{\mu}))$ as $\mu$ increased within $(0,1)$, i.e. we have
\begin{equation}\label{eq-foleq}
\pi_1(\calj(X,\Omega_{\mu})) = \pi_1(\calj(X))\quad \text{\normalfont{for}}\ \mu \in (0,1).
\end{equation}
Diagram \eqref{d-kron} implies that the symplectic mapping class group is the cokernel of $\nu_{*} : \pi_1(\diff_0(X)) \to \pi_1(\calj(X))$ which we now show is trivial.
\begin{lem}\label{lem-epiclemma}
$\nu_{*}$ is an epimorphism.
\end{lem}
\begin{proof}
Though the map $\nu : \diff_{0}(X) \to \calj(X)$ is not a fibration, it can be extended to one; namely, to
\[
\diff_{0}(X) \to \calj(X) \to \fol_0(X),
\]
where the last arrow is a homotopy equivalence, see \slsf{Theorem \ref{thm-ruled}} and \slsf{Proposition \ref{prop-dasha}}. Thus, we end up with the homotopy exact sequence
\[
\ldots \to \pi_1(\diff_0(X)) \to \pi_1(\calj(X)) \to \pi_0(\mathcal{D} \cap \diff_0(X)).
\]
If $X$ is of genus 1, the group $\pi_0(\mathcal{D} \cap \diff_0(X))$ is trivial by \slsf{Propositions} \ref{prop-spin} and \ref{prop-nonspin}. This finishes the proof. \qed
\end{proof}
The following corollary will not be used in the remainder of the paper, but it is a very natural application of \lemma{lem-epiclemma}.
\begin{corol}
The space $\calj(X)$ is homotopy simple. In other words, $\pi_1(\calj(X))$ is abelian and acts trivially on $\pi_{n}(\calj(X))$.
\end{corol}
By virtue of \eqref{d-kron} and \eqref{eq-foleq}, \lemma{lem-epiclemma} immediately implies 
\begin{prop}\label{prop-onehalf}
$\pi_0(\symp^*(X,\omega_{\mu})) = 0$ for every $\mu \in (0,1)$. 
\end{prop}

In order to compute the group $\pi_0(\symp^*(X,\omega_{\mu}))$ for $\mu \in (-1,0)$ it is necessary to know 
better the fundamental group of $\calj_{\cala}$.
The space $\calj_{\cala}$ is the complement to (the closure of) the elliptic divisorial locus $\cald_{\bft_{-}}$ in the ambient space $\calj(X)$. We denote by $i$ the inclusion 
\[
i : \calj_{\cala}(X) \to \calj(X).
\]
By \slsf{Lemma \ref{lem-epiclemma}} every loop $J(t) \in \pi_1(\calj_{\cala})$ can be decomposed into a product $J(t) = J_0(t) \cdot J_1(t)$, where $J_0(t) \in \im \nu_{*}$, and $J_1(t) \in \ker i_{*}$. 

By \lemma{lem-suka}, for $\mu \in (-1,0]$, $\calj_{\cala} = \calj(X,\Omega_{\mu})$. 
In particular, assumption $\sf{(A)}$ is satisfied for each loop in $\calj_{\cala}$. 
Thus, each loop in $\calj_{\cala}$, that lies in $\ker i_{*}$, could contribute drastically to $\pi_0(\symp^*(X,\omega_{\mu}))$ via the corresponding elliptic twists. 
But this will not happen, because the following holds.
\begin{lem}\label{lem-unexpectedlemma}
$\ker i_{*}\subset \im \nu_{*}$.
\end{lem}
\begin{proof}
Choose some $J_{*} \in \cald_{\bft_{-}}$, and let $\Delta$ be a $2$-disc which intersects $\cald_{T_{-}}$ transversally at the single point $J_{*}$. Denote by $J(t)$ the boundary of $\Delta$. By \slsf{Lemma \ref{lem-basepnt}} one simply needs to show that the homotopy class of $J(t)$ comes from the natural action of $\diff_{0}(X)$ on $\calj_{A}$, and the lemma will follow.

If $J_{*}$ is integrable, then one can choose $\Delta$ such that $J(t)$ is indeed an orbit of the action of a certain loop in $\diff_0(X)$, see the description of the complex-analytic family constructed in \refsubsection{sec-family}. Thus it remains to check that every structure $\calj_{*} \in \cald_{\bft_{-}}$ can be deformed to be integrable through structures on $\cald_{\bft_{-}}$. 
This will be proved by \slsf{Lemma \ref{lem-connectinglemma}} below.
\qed
\end{proof}

\begin{lem}\label{lem-basepnt}
Let $x,y \in \calj_{\cala}$, and let $H(t) \in \calj_{\cala},\,t \in [0,1]$ be a path joining them such that $H(0) = x,\, H(1) = y$. If a loop $J(t) \in \pi_1(\calj_{\cala},y),\, t \in [0,1]$ lies in the image of $\pi_1(\diff_0(X),\text{\slsf{id}}) \to \pi_1(\calj_{\cala}, y)$, then $H^{-1} \cdot J \cdot H \in \pi_1(\calj_{\cala}, x)$ lies in the image of $\pi_1(\diff_0(X),\text{\slsf{id}}) \to \pi_1(\calj_{\cala}, x)$.
\end{lem}
\begin{proof}
Without loss of generality we assume that there exists a loop $f(t) \in \pi_1(\diff_0,\text{\slsf{id}})$ such that $J(t) = f_{*}(t) J(0)$.
Let $H_s$ be the piece of the path $H$ that joins the points $H(0) = x$ and $H(s)$. To prove the lemma it remains to consider the homotopy 
\[
J(s,t) := H^{-1}_s \cdot f_{*}(t)H(s) \cdot H_s,
\]
where $J(1,t) = H^{-1} \cdot J \cdot H$ and $J(0,t) = f_{*}(t)H(0)$. \qed
\end{proof}

\begin{lem}\label{lem-connectinglemma}
Every connected component of $\calj_{-1}$ contains at least one integrable structure. \end{lem}
\begin{proof}
Take a structure $J \in \calj_{-1}$, and denote by $C$ the corresponding smooth elliptic curve in the class $[C] = \bft_{-}$.
Let $\pi \colon X \to C$ be the ruling such that $J \in \calj(X,\pi)$, see \slsf{Theorem \ref{thm-ruled}}.
Apart from the section given by $C$, we now choose one more smooth section $C_1$ of $\pi$ such that $C_1$ is disjoint from $C$; the section $C_1$ need not be holomorphic, but be smooth.
We claim that there exists a unique $\cc^{*}$-action on $X$ such that
\begin{itemize}
\item[(a)] it is fiberwise, i.e. this diagram
\begin{equation}
\displaystyle{
\begin{CD}
\scrx @>{\cdot g}>> \scrx \\
@V{p}VV @VV{p}V \\
\cc @>>{\cdot g}> \cc,
\end{CD}
}
\end{equation}
commute for each $g \in \cc^{*}$,
\item[(b)] it acts on the fibers of $\pi$ by means of biholomorphisms, and
\item[(c)] it fixes both $C$ and $C_1$. 
\end{itemize}

The complement $X - C_1$ is a $\cc$-bundle with $C$ being the zero-section; we keep the notation $\pi$ for the projection $X - C_1 \to C$. This bundle inherits the $\cc^{*}$-action described above. Let $\uone$ be the unitary subgroup of $\cc^{*}$. The $(0,1)$-part of a $\uone$-invariant connection on the $\cc$-bundle $\pi \colon X-C_1 \to C$ defines a $\bar\del$-operator which associated to some holomorphic structure $J_1$ on $X-C_1$, see \slsf{Chapter 0, § 5} in \cite{Gr-Ha}.
As a complex structure, $J_1$ agrees with $J$ on the fibers of $\pi$.

To every holomorphic $\cc^{*}$-bundle one canonically associates a $\cp^1$-bundle. Hence, 
there is a unique extension $J_1$ to a complex structure $J_1 \in \calj(X,\pi)$
such that $C_1$ becomes holomorphic.

When restricted to the bundle $TX|_{C}$, $J_1$ coincides with $J$.
By \slsf{Proposition \ref{prop-dasha}} there is a symplectic form $\omega$ taming both structures $J$ and $J_1$. Given a symplectic curve, say $C$, in $X$, and an almost-complex structure, say $J$, defined along $C$ (i.e. on $TX|_{C}$) and tamed by $\omega$. There exists an $\omega$-tamed almost-complex structure on $X$ which extends the given one. Moreover, such an extension is homotopically unique. In particular, one can always construct a family $J_t$ joining $J$ and $J_1$ such that $C$ stays $J_t$-holomorphic, and the lemma is proved. \qed
\end{proof}

\smallskip%
Summarizing the results of \slsf{Lemma \ref{lem-epiclemma}} and \slsf{Lemma \ref{lem-unexpectedlemma}} we obtain
\begin{lem}
$\pi_1(\diff_0(X)) \to \pi_1(\calj_{\cala}(X))$ is epimorphic.
\end{lem}
Again, it is implied by diagram \ref{d-kron} that the following holds.
\begin{prop}\label{prop-2half}
$\pi_0(\symp^*(X,\omega_{\mu})) = 0$ for every $\mu \in (-1,0]$. 
\end{prop}
Together with \slsf{Proposition \ref{prop-onehalf}}, this statement covers what is claimed in \slsf{Theorem \ref{thm-vanish}}.

\section{Blow up once} 

\subsection{Rational $(-1)$-curves.}
Let $(Z,\omega)$ be a symplectic ruled 4-manifold diffeomorphic to $S^2 \tilde{\times} Y^2 \# \overline{\cp}^2$.
Here we study homology classes in $\sfh_2(Z;\zz)$ that can be represented by a symplectically embedded $(-1)$-sphere. Given a symplectically embedded $(-1)$-sphere $A$, it satisfies 
\begin{equation}\label{eq-excurve}
\displaystyle{
[A]^2 = -1, \quad K^{*}([A]) = 1.
}
\end{equation}
A simple computation shows that there are only two homology classes satisfying \eqref{eq-excurve}, namely, $[A] = \bfe$ and $[A] = \bff - \bfe$.

The following lemma will be used in the sequel, often without any specific reference.
\begin{lem}\label{lem-martin}
Let $(Z,\omega)$ be a symplectic ruled 4-manifold diffeomorphic to $S^2 \tilde{\times} Y^2 \# \overline{\cp}^2$.
Then for every choice of $\omega$-tamed almost-complex structure $J$, both the classes $\bfe$ and $\bff - \bfe$ are represented by smooth rational $J$-holomorphic curves.
\end{lem}
\begin{proof}
Given an arbitrary $\omega$-tamed almost-complex structure $J$, the exceptional class $\bff - \bfe$ is represented by either a smooth $J$-holomorphic curve or by a $J$-holomorphic cusp-curve $A$ of the form $A = \sum m_i A_i$ where each $A_i$ stands for a rational curve occuring with the multiplicity $m_i \geq 1$. Clearly, we have
\begin{equation}\label{eq-areadec}
\displaystyle{
0 < \int_{A_i} \omega < \int_{A} \omega.
}
\end{equation}
Because $K^{*}(\bff - \bfe) = 1$, there exists at least one irreducible component of the curve $A$, say $A_1$, with $K^{*}([A_1]) \geq 1$.

Note that spherical homology classes in $\sfh_2(Z;\zz)$ are generated by $\bff$ and $\bfe$. Hence, we have $[A_1] = p \bff - q \bfe$, which implies in particular that $[A_1]^2 = -q^2 \leq 0$, with equality iff $[A_1] = p \bff$. But the latter is prohibited by \eqref{eq-areadec} because 
\[
\displaystyle{
\int_{\bfe} \omega > 0                                      .
}
\]
Therefore, we have $[A_1]^2 \leq -1$. Further, one may use the adjunction formula to obtain that $A_1$ is a smooth rational curve with $[A_1]^2 = -1$ and $K^{*}(A_1) = 1$. Note that it is not possible for $A_1$ to be in the class $\bff - \bfe$ because of \eqref{eq-areadec}. Hence, we have $[A_1] =  \bfe$. 

Take another irreducible component, say $A_2$. If $A_2$ does not intersect $A_1$, then $[A_2] = p \bff$, which contradicts \eqref{eq-areadec}. Thus $A_2$ intersects $A_1$, positively. Hence, $[A_2] = p \bff - q \bfe$ for $q$ positive. The same argument works for the other irreducible components $A_2, A_3, \ldots$ of the curve $A$. But note that $[A_2]\cdot [A_3] < 0$, and hence there are no other components of $A$, except $A_1$ and $A_2$. We thus have $m_2 [A_2] = \bff - (m_1 + 1) \bfe$ for $m_1, m_2 \geq 1$. The class $\bff - (m_1 + 1) \bfe$ is primitive, and hence $m_2 = 1$. Further, this class cannot be represented by a rational curve, which can be easily checked using the adjunction formula. We thus proved the lemma for the class $\bff - \bfe$; the case of $\bfe$ is analogous. \qed
\end{proof}

This lemma leads to the following generalization of \slsf{Theorem \ref{thm-ruled}} for ruled but not geometrically ruled symplectic 4-manifolds.
\begin{lem}\label{lem-ruled2}
Let $(Z,\omega)$ be a symplectic ruled 4-manifold diffeomorphic to $S^2 \tilde{\times} Y^2 \# \overline{\cp}^2$, and let $J$ be an $\omega$-tamed almost-complex structure. Then $Z$ admits a \slsf{singular ruling} given by a proper projection $\pi : Z \to Y$ onto $Y$ such that

\sli there is a singular value $y^{*} \in Y$ such that $\pi$ is a spherical fiber bundle over $Y - y^{*}$, and each fiber $\pi^{-1}(y)$, $y \in Y - y^{*}$, is a $J$-holomorphic smooth rational curve in the class $\bff$;

\slii the fiber $\pi^{-1}(y^{*})$ consists of the two exceptional $J$-holomorphic smooth rational curves in the classes $\bff-\bfe$ and $\bfe$.
\end{lem}
\begin{proof}
It follows from \lemma{lem-martin} that $Z$ admits $J$-holomorphic $(-1)$-curves $E$ and $E^{\prime}$ in the classes
$\bfe$ and $\bff - \bfe$, respectively.
We have to show that for each point $p \in Z$, except for those on $E$ and $E^{\prime}$, 
there exists a smooth $J$-holomorphic sphere in the class $\bff$ that passes through $p$.
Such a sphere would necessarily be unique due to positivity of intersections.
To get such a curve for a generic (by Gromov compactness, for every) point $p \in Z$ it suffices to show $\gr(\bff) \neq 0$. But this follows from the Seiberg-Witten theory, see \cite{McD-Sa-2}.

Having a $J$-holomorphic curve passing through $p \in Z$, we have to prove it is smooth. 
Along the same lines as \lemma{lem-martin}, one shows that the only non-smooth $J$-holomorphic curve in the class $\bff$ is $E \cup E^{\prime}$.  \qed
\end{proof}

\subsection{Straight structures.}\label{sec-straight} Let $Z \cong S^2 \tilde{\times} T^2 \# \overline{\cp}^2$ be a complex ruled surface, and let $E$ be a smooth rational $(-1)$-curve in $\bfe \in \sfh_2(Z;\zz)$. The blow-down of $E$ from $Z$, which is a non-spin geometrically ruled genus one surface, will be denoted by $X$.
The surface $Z$ is said to be a \slsf{type $\cala$ surface} if $X$ is biholomorphic to the surface $X_{\cala}$, see \refsubsection{sec-complex}.

Let $p \in X$ be the image of $E$ under the contraction map. Recall that $X_{A}$ contains the triple of bisections, which are smooth elliptic curves in the class $\bfb \in \sfh_2(X;\zz)$. The surface $Z$ is called  \slsf{straight type $\cala$ surface} if there is no bisection passing through $p$ in $X$. In other words, a straight type $\cala$ surface contains a triple of smooth curves in the homology class $\bfb$, while a non-straight type $\cala$ surface contains a smooth elliptic $(-1)$-curve in the class $\bfb - \bfe \in \sfh_2(Z;\zz)$. We remark that it follows from \refthm{thm-fibering} that straight type $\cala$ surfaces can be characterized as those for which there exists a smooth elliptic $(-1)$-curve in the homology class $2\bfb-\bfe \in \sfh_2(Z;\zz)$.

Let $\pi$ be the ruling of $X$, and let $S$ be the fiber of $\pi$ that passes through $p$.
When $Z$ is type $\cala$, there are three bisections $B_i \subset X$, each of which intersects $S$ at precisely two distinct points. The following result was established in \refsubsection{sec-complex}, see the construction of Suwa's model.
\begin{lem}
There exists a complex coordinate $s$ on $S$ such that the intersection points $B_i \cap S$ are as follows:
\begin{equation}\label{eq-bisectionpoints}
B_1 \cap S = \left\{ 0, \infty \right\},\quad B_2 \cap S = \left\{ -1, 1 \right\},\quad B_3 \cap S = \left\{ -i, i \right\}. 
\end{equation}

\end{lem}
We then claim
\begin{lem}\label{lem-familycp}
There exists a complex-analytic family $\scrz \to \cp^1$ of type $\cala$ surfaces $Z_s$ parametrized by $s \in \cp^1$.
When $s$ equals one of the exceptional values
\[
\left\{ 0, \infty \right\},\quad \left\{ -1, 1 \right\},\quad \left\{ -i, i \right\},
\]
the surface $Z_s$ is not a straight type $\cala$ surface, while for other parameter values, $Z_s$ is straight type $\cala$.
\end{lem}
\begin{proof}
Pick a fiber $F$ of the ruling of $X \cong X_{\cala}$.
Consider the complex submanifold $F \times \cp^1 \subset X \times \cp^1$, and denote by $S$ the diagonal in $F \times \cp^1$. 
We construct $\scrz$ as the blow-up of $X \times \cp^1$ along $S$. The $3$-fold $\scrz$ forms the complex-analytic family $\scrz \to S$ that was claimed to exist in the lemma. \qed 
\end{proof}

\medskip%
The notion of the straight type $\cala$ complex structure can be generalized to almost-complex geometry as follows. Choose a tamed almost-complex structure $J \in \calj(Z)$. We will call $J$ straight type $\cala$, or simply \slsf{straight}, if each $J$-holomorphic representative in the class $\bfb \in \sfh_2(Z;\zz)$ is smooth. Clearly, the space of straight structures $\calj_{\text{st}}(Z)$ is an open dense submanifold in $\calj(Z)$. Instead of $\calj(Z)$ or $\calj_{\text{st}}(Z)$ we write $\calj$ and $\calj_{\text{st}}$ for short. This definition of straightness is motivated by the following lemma the proof of which is left to the reader because it is similar to the proof of \slsf{Proposition \ref{prop-smoothcurves}}, (but the modified version of \slsf{Theorem \ref{thm-ruled}} given by \slsf{Lemma \ref{lem-ruled2}} should be used). 

Let $\mib{s} \colon [0,1] \to S$ be a loop in $S$, and let 
$Z(t) \to \mib{s}(t)$ be the restriction of $\scrz \to S$ to $\mib{s}(t)$. Because $\mib{s}(t)$ is contractible inside the sphere $S$, we can think of $Z(t)$ as a family of type $\cala$ complex structures on $Z$. Each structure $J(t)$ is straight iff $\mib{s}(t)$ 
does not pass through any of points \eqref{eq-bisectionpoints}.
The following choice of $\mib{s}(t)$ will be used in \refsubsection{sec-cocycle}
\begin{equation}\label{refereemustdie}
\mib{s}(t) = \varepsilon e^{2 \pi i t }.
\end{equation}
\begin{lem}\label{lem-straightness}
Let $(Z,\omega)$ be a symplectic ruled 4-manifold diffeomorphic to $S^2 \tilde{\times} T^2 \# \overline{\cp}^2$, and let $J$ be an $\omega$-tamed almost-complex structure. Then every $J$-holomorphic representative in the class $\bfb$ is either irreducible smooth or contains a smooth component in one of the classes $\bfb - \bfe$, $\bft_{+}-k\bff$, $k > 0$. 
\end{lem}
Similarly to \slsf{Proposition \ref{prop-smoothcurves}} this lemma leads to a natural stratification of the space $\calj$ of tamed almost-complex structures. Namely, this space can be presented as the disjoint union
\[
\calj = \calj_{\text{st}} \bigsqcup \cald_{\bft_{-}} \bigsqcup \cald_{\bfb - \bfe} + \ldots,
\]
where $\cald_{\bft_{-}}$ and $\cald_{\bfb - \bfe}$, which are submanifolds of real codimension 2 in $\calj$, are the elliptic divisorial locus for respectively the classes $\bft_{-}$ and $\bfb-\bfe$. Here we omitted the terms of real codimension greater than 2, because they do not affect the fundamental group of $\calj$.

Coming to the symplectic side of straightness, we claim that if a symplectic form $\omega$ on $Z$ satisfies the \slsf{period conditions}
\[
\displaystyle{
\int_{\bft_{-}} \omega < 0, \quad \int_{\bfb - \bfe} \omega < 0,
}
\]
then $\calj(Z,\omega) \subset \calj_{\text{st}}$. Moreover, a somewhat inverse statement holds, at least for integrable structures.
\begin{lem}[cf. \refsubsection{Ell-tw-1}]\label{lem-kahst}
Every complex straight type $\cala$ structure is tamed by a symplectic form satisfying the period conditions.
Moreover, every compact family of straight type $\cala$ structures is tamed by a family of cohomologous forms satisfying the period conditions.
\end{lem}
\begin{proof}
We first check that a complex straight type $\cala$ surface $Z$ has a taming symplectic form $\theta$ such that $\theta$ satisfies the period conditions.

If $Z$ is type $\cala$ then it is the surface $X_{\cala} \cong S^2 \tilde{\times} T^2$ blown-up once. Since $X_{\cala}$ admits a symplectic structure which satisfies the first period condition, then so does $Z$. Further, the second period condition can be achieved by means of deflation along a smooth elliptic curve in the class $2\bfb - \bfe$; such a curve indeed exists thanks to the straightness of $Z$.

We let $K$ to denote the parameter space for our family $Z_t$, and let $\theta_t$, $t \in K$ be a taming symplectic form on $Z_t$ that satisfies the period condition. For every point $t^\prime \in K$, let $U_{t^\prime} \in K$ be a sufficiently small neighbourhood of $t^\prime \in K$ such that for each $t \in U_{t^\prime}$ 
\[
\text{\normalfont{ $\theta_{t^\prime}$ tames the complex structure in $Z_t$. }}
\]

As $K$ is compact, one may take a finite subcover $U_{t_i}$, $t_i \in I$ of $K$.
The forms $\theta_{I}$ are not necessarily cohomologous because they may have different integrals on the homology class $\bfe$. Set $\varepsilon_{t_i} := \int_{\bfe} \theta_{t_i}$, $t_i \in I$, and $\varepsilon := \min\, \varepsilon_{t_i}$. We now deflate $(Z_{t_i}, \theta_{t_i})$ along the homology class $\bfe$ to get $\int_{\bfe} \theta_{t_i} = \varepsilon$. Thanks to this deflation the forms $\theta_{I}$ become cohomologous and still do satisfy the period conditions.

Finally, set $\hat{\theta}(t) := \sum_{I^{\prime}} \rho_{t_i}(t) \theta_{t_i}$, where the functions $\rho_{t_i} = \rho_{t_i}(t)$ is a partition of unity for the finite open cover $U_{t_{i}}, {t_{i}} \in I$ of $K$. What remains is to verify that $Z_t$ is tamed by $\hat{\theta}(t)$ for every $t \in K$. Pick some $t^{*} \in K$, then there are but finitely many charts $U_{t_{1}}, \ldots, U_{t_{p}}$ that contains the point $t^{*} \in K$. Then $\hat{\theta}(t^{*}) = \rho_{t_1}(t^*) \theta_{t_1} + \ldots + \rho_{t_p}(t^*) \theta_{t_p}$. Since each of $\theta_{t_1}, \ldots, \theta_{t_p}$ tames $Z_{t^{*}}$, then so does $\hat{\theta}(t^{*})$.
\qed
\end{proof}
\smallskip%

From this we have verified assumption $\sf{(A)}$ for the family of straight structures given by \eqref{refereemustdie}.

\subsection{Refined Gromov invariants.}\label{sec-newmoduli}
In this subsection, we work with an almost-complex manifold $(Z,J)$ equipped with a straight structure $J \in \calj_{\text{st}}$, i.e. every $J$-holomorphic curve of class $\bfb \in \sfh_2(Z;\zz)$ in $Z$ is smooth. We also note that such a curve is not multiply-covered, because the homology class $\bfb$ is primitive. The universal moduli space $\calm(\bfb;\calj_{\text{st}})$ of embedded non-parametrized pseudoholomorphic curves of class $\bfb$ is a smooth manifold, and the natural projection $\pr : \calm(\bfb;\calj_{\text{st}}) \to \calj_{\text{st}}$ is a Fredholm map, see \cite{Iv-Sh-1, McD-Sa-3}. 
Given a generic $J \in \calj_{\text{st}}$, the preimage $\pr^{-1}(J)$ is canonically oriented zero-dimensional manifold, see \cite{Tb} where it is explained how this orientation is chosen. 
The cobordism class of $\pr^{-1}(J)$ is independent of a generic $J$, thus giving us a well-defined element of $\Omega_0^{\mib{SO}} = \zz$; the number is equal to $\gr(\bfb)$. 

\slsf{Corollary \ref{cor-gw}} states that $\gr(\bfb) = 3$, and hence $Z$ contains not one but several curves in the class $\bfb$. Once we restricted almost-complex structures to those with the straightness property, the following modification of Gromov invariants can be proposed: given the image $G$ of a certain homomorphism $\zz^2 \to \sfh_1(Z;\zz)$, instead of counting pseudoholomorphic curves $C$ such that $[C] = \bfb$, we will count curves $C$ such that $[C] = \bfb$ and the embedding $i : C \hookrightarrow Z$ satisfies $\im i = G$.  
The definitions of Gromov invariants 
$\gr(\bfb,G)$, moduli space $\calm(\bfb,G;\calj_{\text{st}})$, and so forth are completely analogous to those in ``usual'' theory of Gromov invariants.

Suppose $J$ is an integrable straight type $\cala$ structure, then the complex surface $(Z,J)$ contains precisely 3 smooth elliptic curves $C_1$, $C_2$, and $C_3$ in the homology class $\bfb$. We denote by $G_k$ the subgroup of $\sfh_1(Z;\zz)$ generated by cycles on $C_k$; these subgroups $G_k$ are pairwise distinct, see \refsubsection{sec-complex}.

It is clear now that the space $\calm(\bfb;\calj_{\text{st}})$ is disconnected and can be presented as the union
\[
\displaystyle{
\calm(\bfb;\calj_{\text{st}}) = 
\bigsqcup\limits_{k=1}^{3} \calm(\bfb,G_{k};\calj_{\text{st}
}). 
}
\]
We define the \slsf{moduli space of bisections} to be the fiber product
\[
\displaystyle{
\calm_{3B} = 
\left\{ (x_{1}, x_{2}, x_{3})\ |\  x_{k} \in \calm(\bfb,G_{k};\calj_{\text{st}
}),\ \pr(x_1) = \pr(x_2) = \pr(x_3) ) \right\}. 
}
\]
Similarly to $\calm(\bfb;\calj_{\text{st}})$, the moduli space $\calm_{3B}$ is a smooth manifold and the projection $\pr : \calm_{3B} \to \calj_{\text{st}}$
is a smooth map. We close this section by stating an obvious property of the projection map that we shall use in the sequel. 
\begin{lem}\label{lem-diffcomplex}
The projection map $\pr : \calm_{3B} \to \calj_{\text{st}}$ is a diffeomorphism, when is restricted to the subset of integrable straight type $\cala$ complex structures.
\end{lem}

\subsection{Loops $\calm_{3B}$.}\label{sec-cocycle} 

The map $\nu : \diff_0(Z) \to \calj_{\text{st}}$ defined by
\[
\diff_{0}(Z) \xrightarrow{\nu} \calj(Z,\omega) : f \to f_{*} J,
\]
can be naturally lifted to a map $\diff_0(Z) \to \calm_{3B}$. Indeed, take a point $\mib{s} \in \calm_{3B}$, which is a quadruple
$[J, B_1, B_2, B_3](\mib{s})$ consisting of an almost-complex structure $J(\mib{s}) \in \calj_{\text{st}}$ on $Z$ and a triple of smooth $J(\mib{s})$-holomorphic elliptic curves $B_1(\mib{s})$, $B_2(\mib{s})$, and $B_3(\mib{s})$ in $Z$. Then one can define
\[
\displaystyle{
\diff_{0}(Z) \xrightarrow{\nu} \calm_{3B} : f \to [f_{*} J, f(B_1), f(B_2), f(B_3)].
}
\]
Here we construct an element of $\pi_1(\calm_{3B})$ that does not lie in the image of $\nu_{*} \colon \pi_1(\diff_0(Z)) \to \pi_1(\calm_{3B})$.
\medskip%

To start, we consider the tautological bundle $\calz \cong \calm_{3B} \times Z$ over $\calm_{3B}$ whose fiber over a point $x \in \calm_{3B}$ is the almost-complex manifold $(Z,J(\mib{s}))$.
By \slsf{Lemma \ref{lem-martin}} every almost-complex manifold $(Z,J(\mib{s}))$ contains a unique smooth rational $(-1)$-curve $S(\mib{s})$ in the class $\bff-\bfe$.  
Thus, one associates to $\calz$ an auxiliary bundle $\cals$ whose fiber over $x \in \calm_{3B}$ is the rational curve $S(\mib{s})$.
Note that each $B_i(\mib{s})$ intersects $S(\mib{s})$ at precisely 2 distinct points denoted by $P_{i,1}$ and  $P_{i,2}$. Hence we can mark out 3 distinct pairs of points $(P_{i,1}, P_{i,2})$, $i = 1,2,3$ on each fiber $S(\mib{s})$ of $\cals$.
Besides that, every $(Z,J(\mib{s}))$ contains a unique smooth rational curve $E(\mib{s})$ in the class $\bfe$. The curve $E(\mib{s})$ intersects $S(\mib{s})$ at precisely one point, say $Q(\mib{s})$. This point $Q$ does not coincide with any of the point $P_{i,1},P_{i,2}$, because $J(x)$ is assumed to be a straight one.
Therefore $\cals$ can be considered as a fiber bundle over $\calm_{3B}$ whose fiber is the rational curve $S(\mib{s})$ with 7 distinct marked points, partially ordered as 
\begin{equation}\label{order}
    \left( \left\{ P_{1,1}, P_{1,2} \right\}, \left\{ P_{2,1}, P_{2,2} \right\}, \left\{ P_{3,1}, P_{3,2} \right\}, Q  \right)
\end{equation}
As such, there is an obvious map
\[
\lambda \colon \calm_{3B} \to \scrm
\]
sending $S(\mib{s})$ to the corresponding point in the moduli space of $7$ points in $\cp^1$, partially ordered as 
\eqref{order}. Notice that the space $\scrm$ is also the moduli space of $6$ points on $\cc$, partially ordered as 
$\left( \left\{ P_{1,1}, P_{1,2} \right\}, \left\{ P_{2,1}, P_{2,2} \right\}, \left\{ P_{3,1}, P_{3,2} \right\}\right)$.
One considers $\scrm$ as a quotient $\conf_6(\cc)/\Aff(1,\cc)$, where $\conf_6(\cc)$ is the configuration space of sextuples $(z_1,\ldots,z_6) \in \cc^6$, $z_i \neq z_j$ with the identifications 
\[
(z_1,z_2, \ldots) \sim (z_2,z_1, \ldots),\ (\ldots, z_3,z_4,\ldots) \sim (\ldots, z_4,z_3,\ldots),\ (\ldots, z_5, z_6 ) \sim (\ldots, z_6, z_5 ).
\]
The homotopy exact sequence for $\conf_6(\cc) \to \scrm$ reads
\[
\zz \cong \pi_1(\Aff(1,\cc)) \xrightarrow{} \pi_1(\conf_6(\cc)) \xrightarrow{} \pi_1(\scrm) \xrightarrow{} 1 = \pi_0(\Aff(1,\cc)).
\]
Let $\delta^2$ be the element of $\pi_1(\conf_6(\cc))$ coming from $\pi_1(\Aff(1,\cc))$. It is known that $\delta^2$ generates the center of $\pi_1(\conf_6(\cc))$ (and even the center of a larger group, the braid group on $6$ strands.)

Let $\mib{\gamma} \colon [0,1] \to \scrm$ be the loop given by 
\[
(P_{1,1}(t),P_{1,2}(t),\,P_{2,1}(t),P_{2,2}(t),\,P_{3,1}(t),P_{3,2}(t),\, Q(t)) = (0,\infty,\,1,-1,\,i,-i,\,\varepsilon\, e^{2 \pi i t})
\]
with respect to some inhomogeneous coordinate on $\cp^1$. 
Introducing the transformation 
\[
z \to \dfrac{\varepsilon e^{2 \pi i t} z - 1}{ \varepsilon e^{2 \pi i t} - z },
\]
in which $Q(t) = \infty$, one lifts $\mib{\gamma}$ to $\conf_6(\cc)$ as
\[
z_1(t) = -\dfrac{1}{\varepsilon e^{2 \pi i t}},\  z_2(t) = - \varepsilon e^{2 \pi i t},\  
z_3(t) = 1,\  z_4(t) = -1, 
\]
\[
z_5(t) = \dfrac{\varepsilon e^{2 \pi i t} i - 1}{ \varepsilon e^{2 \pi i t} - i },\  
z_6(t) = -\dfrac{\varepsilon e^{2 \pi i t} i + 1}{ \varepsilon e^{2 \pi i t} + i }.
\]
It is not hard to show that the homology class of this loop is non-zero in $\sfh_1(\conf_6(\cc));\rr)$ and not a multiple of $\delta^2$. Following \cite{Ar}, one can prove this by integrating 
the differential form
\[
\mib{\alpha} := \dfrac{1}{2 \pi i}\dfrac{\text{\normalfont{d}}\,z_1 - \text{\normalfont{d}}\,z_2}{z_1 - z_2} - 
\dfrac{1}{2 \pi i}\dfrac{\text{\normalfont{d}}\,z_3 - \text{\normalfont{d}}\,z_4}{z_3 - z_4},
\]
for which $\int_{\delta^2} \mib{\alpha} = 0$ yet $\int_{\gamma} \mib{\alpha} = -1$.
As such, one obtains: $[\mib{\gamma}] \neq 0$ in $\sfh_1(\scrm;\rr)$.
\smallskip%

Using the family $\scrz \to \cp^1$ from \lemma{lem-familycp}, we get a
loop $\mib{s} \colon [0,1] \to \calm_{3B}$ with $\lambda(\mib{s}(t)) = \mib{\gamma}(t)$. The class $[\mib{s}] \in \pi_1(\calm_{3B})$ does not lie in $\im \nu_{*}$, as for it were, that would imply that $[\mib{s}] \in \ker \lambda_*$. (Here we used the inclusion $\im \nu_{*} \subset \ker \lambda_*$ following from the fact that $\lambda$ is $\diff_0(Z)$-invariant.) 
\smallskip%

If $\mib{\gamma}$ was given either by 
\[
(P_{1,1}(t),P_{1,2}(t),\,P_{2,1}(t),P_{2,2}(t),\,P_{3,1}(t),P_{3,2}(t),\, Q(t)) = (0,\infty,\,1 + \varepsilon e^{2 \pi i t},-1,\,i,-i,\,0),\ 
\text{\normalfont{or}}
\]
\[
(P_{1,1}(t),P_{1,2}(t),\,P_{2,1}(t),P_{2,2}(t),\,P_{3,1}(t),P_{3,2}(t),\, Q(t)) = (0,\infty,\,1,-1,\,i + \varepsilon e^{2 \pi i t} ,-i,\,0),
\]
then a similar argument would work to get another non-contractible loop in $\calm_{3B}$.

\subsection{Loops in $\calj_{\text{st}}$.}\label{sec-loops}
Here we construct an element of $\pi_1(\calj_{\text{st}})$ that does not lie in the image of $\nu_{*} \colon \pi_1(\diff_0(Z)) \to \pi_1(\calj_{\text{st}})$.
\medskip%

Let $J(t)$ be the loop of integrable structures with $\pr(\mib{s}(t))=J(t)$ for the loop $\mib{s}(t)$ constructed in \refsubsection{sec-cocycle}.

\begin{lem}\label{lem-key1}
The class $[J] \in \pi_1(\calj_{\text{st}})$ does not lie in $\im \nu_{*}$. 
\end{lem}
\begin{proof}
Assume the contrary, i.e. that there exists a family $f \colon [0,1] \to \diff_0(Z)$, $f(0) = f(1) = \id$ such that $J(t)$ is homotopic to $\wt{J}(t) := f(t)_{*} J(0)$. We join $J(t)$ and $\wt{J}(t)$ with a tube $T \subset \calj_{\text{st}}$.
By Sard-Smale theorem we can arrange that $T$ is transverse to $\pr$. Thus, the preimage $\pr^{-1}(T)$ is a smooth orientable surface that bounds $\mib{s}(t) \cup \wt{\mib{s}}(t)$. 
Note that $\wt{\mib{s}}(t) := \pr^{-1}(\wt{J}(t))$ is connected thanks to \lemma{lem-diffcomplex}. It follows that $[\wt{\mib{s}}] = [\mib{s}]$ in $\sfh_1(\calm_{3B};\zz)$. This is a contradiction, as $[\mib{s}]$ does not lie in $\ker \lambda_{*}$, whereas $\lambda$ itself is constant on $\wt{\mib{s}}(t)$. \qed
\end{proof}

\subsection{Let's twist again.}
Here we outline the proof of \slsf{Theorem} \ref{thm-notvanish}, referring the reader to the previous subsections for details.

Let $X$ be a type $\cala$ surface, see \refsubsection{sec-complex}.
It follows from \refthm{thm-fibering} hat $X$ contains a triple of smooth elliptic curves $C_1, C_2$, and $C_3$ in the homology class $\bfb \in \sfh_2(X; \mathbb{Z})$, $[\bfb]^2 = 0$.
The procedure given in \refsubsection{Ell-tw-1} shows that there are three elliptic twists for $X \# \overline{\cp}^2$, see also \lemma{lem-kahst}. 
Denote the corresponding loops by $J_{C_1}, J_{C_2}$, and $J_{C_3}$; they are contained in the space $\calj_{\text{st}}$ of the straight almost-complex structures, see \refsubsection{sec-straight}. 
We prove these loops do not lie in the image of $\nu_{*} : \pi_1(\diff_0(X \# \overline{\cp}^2)) \to \pi_1(\calj_{\text{st}})$, see \refsubsection{sec-loops}. 

Using \lemma{lem-kahst}, we find a symplectic form $\theta$ with $\int_{\bfb - \bfe} \theta \leq 0$ such that 
$J_{C_i} \in \calj(X \# \overline{\cp}^2, \Theta)$, $i = 1,2,3$. Here $\Theta$ stands for $\Omega(X.\theta)$.
Since the inclusion $\calj(X \# \overline{\cp}^2, \Theta) \subset \calj_{\text{st}}$ is equivariant w.r.t. to the natural action of $\diff_0(X \# \overline{\cp}^2)$ on these spaces, it follows that $J_{C_i}$ do not lie in the image of $\nu_{*} : \pi_1(\diff_0(X \# \overline{\cp}^2)) \to \calj(X \# \overline{\cp}^2, \Theta)$, and the theorem follows.

\end{document}